\documentclass[leqno]{amsart}
\usepackage{amsthm}
\usepackage{color}
\usepackage{amssymb}
\usepackage{amsmath,amsfonts,enumerate}
\usepackage{amsthm}
\usepackage{hyperref}

\numberwithin{equation}{section}

\newtheorem{thm}{Theorem}
\newtheorem{lem}[thm]{Lemma}

\newtheorem{rmk}[thm]{Remark}

\begin{document}
\title[Optimal distributional estimates of the multiple Hilbert transform]
{Optimal distributional estimates of the multiple Hilbert transform}

\author{F. Sukochev}
\address{School of Mathematics and Statistics, University of New South
	Wales, Kensington, 2052, NSW, Australia}
\email{f.sukochev@unsw.edu.au}
\author{K. Tulenov}
\address{	School of Mathematics and Statistics, University of New South Wales, Kensington, NSW, 2052, Australia;
	Institute of Mathematics and Mathematical Modeling, 050010 Almaty, Kazakhstan.}
\email{tulenov@math.kz}
\author{D. Zanin}
\address{School of Mathematics and Statistics, Central South University Changsha, 410075, People's Republic of China}
\email{d.zanin@csu.edu.cn}


\date{}
\begin{abstract} In this paper, we study optimal distributional estimates for the multiple Hilbert transform. We obtain pointwise upper and lower distributional estimates of the multiple Hilbert transform in terms of the $d$-fold composition of the Calder\'{o}n operator with itself. This extends the fundamental results by A. P. Calder\'{o}n \cite{Calderon-1966},  D.  Boyd  \cite{Boyd}, and Ch. Fefferman \cite{CFefferman1972} for arbitrary $d\in\mathbb N.$
\end{abstract}
\maketitle

\section{Introduction}

For locally integrable functions $f$ on $\mathbb{R},$  the Hilbert transform $H$ is defined (in the principal value sense)  by
\[
Hf(x) \overset{ \mathrm{p.v.} }{=}\int_{\mathbb{R}} \frac{f(y)}{x-y}dy
\overset{ \mathrm{p.v.} }{=}f* \frac{1}{x}=\lim_{\varepsilon\to 0^+} \int_{|y|>\varepsilon} \frac{f(x-y)}{y}\,dy.
\]
Let $S$ be  the Calder\'{o}n operator  \cite{Calderon-1966} (\cite[Chapter III.4, pp. 133-134]{BSh}) defined  by
$$
(Sf)(t):=\frac{1}{t}\int_0^t f(s)\,ds+\int_t^\infty f(s)\frac{ds}{s}, \quad f\in \Lambda_{\log}(\mathbb R_+),
$$
where $\Lambda_{\log}(\mathbb R_+)$ is the Lorentz space (see \eqref{Lorentz-log}) associated with the function $\log(1+t),\,\ t>0.$
For $d\in\mathbb N ,$ let $S^d:=S\circ \cdots \circ S$ denote the $d$-fold composition of $S$ \eqref{S d power} with itself  and $H^{\otimes d}$  \eqref{d-tensor-hilbert tr} be the multiple Hilbert transform defined by 
\begin{equation}\label{Mutiple_HT}
	H^{\otimes d}f(x):=H_{1} H_{2} \cdots H_d f(x)\overset{ \mathrm{p.v.} }{=}f * \frac{1}{x_1 x_2 \cdots x_d}, \,\ x=(x_1,\dots,x_d)\in\mathbb R^d,
\end{equation}
for locally integrable functions $f$ on $\mathbb R^d$  (see \cite{ChangFefferman1985}, \cite{CFefferman1972}). Here, $ H_{j} f(x)\overset{ \mathrm{p.v.} }{=}f * \frac{1}{x_j }, \,\ j=1,\dots,d.$
The following is the main result of the paper, which establishes optimal distributional estimates of the multiple Hilbert transform defined by \eqref{Mutiple_HT}.

\begin{thm}\label{main thm} Let $d\in\mathbb{N}$ and let $f:\mathbb R^d\to \mathbb R$ be a measurable function such that $\mu(f)\in {\rm dom}(S^d).$ We have 
	\begin{enumerate}[{\rm (i)}]
		\item  $\mu(H^{\otimes d}f)\leq c(d)S^d \mu(f),$
		
		where $\mu(f)$ is the non-increasing rearrangement of the function $|f|.$
		\item There exists a sequence $\{f_{N,d}\}_{N\in\mathbb{N}}$  of measurable functions: $f_{N,d}:\mathbb{R}^d\to \mathbb{R}$ such that $\mu(f_{N,d})=\mu(f)\chi_{(0,N^d)}$ and
		$$\liminf_{N\to\infty}\mu(H^{\otimes d}f_{N,d})\geq c'(d)S^d\mu(f)$$
		pointwise.
	\end{enumerate}
	Here,  $c(d)$ and $c'(d)$ are positive constants depending on $d$ only.
\end{thm}
Later we will show that ${\rm dom}(S^d)=\Lambda_{\varphi_d}(\mathbb{R}_+)$ is  the Lorentz space defined by \eqref{main-Lorentz}. This theorem extends the fundamental results by Calder\'{o}n \cite{Calderon-1966} (see \cite[Theorem III. 5.7]{BSh} ), Boyd \cite{Boyd},  \cite[Theorem III. 4.8, p. 138]{BSh}  and \cite[Proposition III. 4.10, p.140]{BSh}, and Ch. Fefferman \cite{CFefferman1972} in the special cases $d=1$ and $d=2,$ respectively, as well as the work of Zhizhiashvili \cite{Zhizhiashvili1963, Zhizhiashvili1973,Zhizhiashvili1974,Zhizhiashvili1985} for arbitrary $d\geq 1.$
The whole paper is devoted to the proof of this theorem. 

\subsection{Main motivation and historical background}

The study of singular integral operators that commute with multiparameter families of dilations reveals a striking and fundamental phenomenon: the number of parameters plays a decisive role, and even small changes in this number can drastically alter the behavior of the operators. Such operators arise naturally in connection with multiple Fourier integrals \cite{Zygmund1968} and multiple Fourier series \cite{Zhizhiashvili1996}, and their analysis highlights intrinsic difficulties of the multiparameter setting. In particular, singular integrals in this context may exhibit higher-dimensional singular sets, placing them beyond the direct reach of the classical Calder\'{o}n--Zygmund theory \cite{CFefferman1972},\cite{ChangFefferman1985}.

Among singular integrals compatible with multiparameter dilations, the most fundamental example is the multiple Hilbert transform $H^{\otimes d}$ (see \eqref{Mutiple_HT}), whose kernel possesses one-dimensional singularity sets. Already in the simplest nontrivial case $d=2$, the double Hilbert transform $H^{\otimes 2}$ exhibits behavior that sharply contrasts with the one-parameter theory. In particular, operators of this type fail to be of weak type $(1,1)$ on $L_1(\mathbb{R}^d)$.

Despite the modern terminology of multiparameter harmonic analysis (or harmonic analysis on product spaces), the origins of the subject date back to the early 1930s, when Jessen, Marcinkiewicz, Zygmund, and others investigated the strong maximal function \cite{JessenMarcinkiewiczZygmund1935}, \cite{SokolSokolowski1947}, \cite{Zygmund1949}. The contemporary framework was developed in the early 1980s through the pioneering work of R. Fefferman, Chang \cite{ChangFefferman1980,ChangFefferman1982,ChangFefferman1985}, and Stein \cite{FeffermanStein1982}, and further advanced by R. Fefferman \cite{Fefferman1987,Fefferman1979}. In particular, the $H^p$ theory on product domains, introduced by Gundy and Stein \cite{GundyStein1979}, provides a natural analogue of the classical one-dimensional theory and is closely connected with operators such as the double Hilbert transform. From the perspective of Fourier analysis, the Hilbert transform plays a decisive role in questions concerning the norm convergence of Fourier series and Fourier integrals. The multiple Hilbert transform arises naturally in the study of rectangular partial sums of multiple Fourier series: while in one dimension partial sums are governed by the Hilbert transform, in higher dimensions the corresponding role is played by $H^{\otimes d}$.

A landmark result of Ch. Fefferman \cite{CFefferman1972} in 1972 demonstrates that even operators with one-dimensional singularities, such as the maximal double Hilbert transform, require fundamentally new ideas. Specifically, such operators fail to satisfy weak type $(1,1)$ estimates. In the same work, Ch. Fefferman proved that the maximal double Hilbert transform
\begin{equation}\label{max-HT}
	H^2_* f(x,y) = \sup_{\varepsilon_1,\varepsilon_2>0}\left| \int\limits_{|x'|>\varepsilon_1}\int\limits_{|y'|>\varepsilon_2} \frac{f(x-x',y-y')}{x'y'}\,dx'dy' \right|
\end{equation}
belongs to $L_{1,\infty}([0,1]\times[0,1])$ for every $f \in L\log L(\mathbb R^2)$ with ${\rm supp}(f)\subset [0,1]\times[0,1]$, where $L\log L(\mathbb R^2)$ is the space of measurable functions  $f$ on $\mathbb{R}^{2}$ satisfying
$$
\int_{\mathbb{R}^{2}} |f(x)|\bigl(1+\log|f(x)|\bigr) dx < \infty,
$$
and $L_{1,\infty}$ denotes the weak-$L_1$ space defined by \eqref{weak-L1}. In particular, the same conclusion holds for the double Hilbert transform $H^{\otimes 2}f$. These results were subsequently extended to the general multiple Hilbert transform $H^{\otimes d}$ by L. Zhizhiashvili \cite{Zhizhiashvili1963, Zhizhiashvili1973,Zhizhiashvili1974,Zhizhiashvili1985,Zhizhiashvili1996}.

If we return to the classical case, a fundamental result due to D. Boyd \cite{Boyd}, \cite[Theorem III.4.8 and Proposition III.4.10]{BSh} establishes that if $f \in \Lambda_{\log}(\mathbb{R})$, then
\begin{equation}\label{eq:1.2}
	\mu(Hf) \leq c_{{\rm abs}} S\mu(f), 
	\quad \mu(f) \in \Lambda_{\log}(\mathbb{R}_+),
\end{equation}
and, moreover, for every such $f$ there exists a measurable function $g$ on $\mathbb{R}$ satisfying $\mu(g)=\mu(f)$ and
\begin{equation}\label{eq:1.3}
	S \mu(f) \leq 2\pi \, \mu(Hg).
\end{equation}
These distributional estimates lie at the heart of weak-type interpolation theory and the Marcinkiewicz interpolation theorem. In particular, Calder\'{o}n’s fundamental weak-type theorem (see \cite[Theorem III.5.7]{BSh}), together with Boyd’s estimates \eqref{eq:1.2}–\eqref{eq:1.3}, yields a characterization of joint weak-type $(p,p;q,q)$ interpolation spaces (see \cite[Chapter III]{BSh}). Furthermore, these inequalities not only establish the existence of the Hilbert transform on $L_p(\mathbb{R})$, $1\le p<\infty$, but also provide a direct route to weak type $(1,1)$ and strong type $(p,p)$ estimates via Hardy’s inequality \cite[Lemma III.3.9]{BSh}. Recent advances have produced analogous distributional estimates for other singular operators, including commutators with Calderón–Zygmund integral operators, Marcinkiewicz multipliers and Littlewood--Paley square functions \cite{SYZZ}.

To the best of our knowledge, no analogue of Boyd’s distributional estimates has been established in the multiparameter setting. The primary objective of this paper is to fill this gap by deriving counterparts of \eqref{eq:1.2} and \eqref{eq:1.3} for the multiple Hilbert transform. Our main result (Theorem \ref{main thm}) provides a complete extension of the Calder\'{o}n--Boyd theory to arbitrary dimension $d\in\mathbb{N}$. The upper distributional estimate (i) in Theorem \ref{main thm} is obtained via an iteration argument combined with techniques developed in \cite[Theorem 14]{STZ_JFA}, while the lower estimate (ii) in Theorem \ref{main thm} requires a substantially more delicate analysis.

In a broader sense, our results significantly strengthen those of Ch. Fefferman \cite{CFefferman1972} for $H^{\otimes 2}$ and of Zhizhiashvili for $H^{\otimes d}$, as it establishes boundedness in general symmetric function spaces, of which the setting considered in their papers appears as a particular case.  
\section{Preliminaries}
Let $d\in \mathbb{N}$ and $I\subset \mathbb{R}^d$ be a measurable set with the Lebesgue measure $m(I)$. We denote $S(I)$ the space of all Lebesgue $m$-measurable functions on $I$ such that $m(\{|f| > s\})$ is finite for some $s > 0.$ Let $L_p(I)$ ($1\leq p<\infty$) be the $L_p$-spaces of pointwise almost-everywhere equivalence classes of $p$-integrable functions in $S(I),$ while $L_{\infty}(I)$ denotes the space of essentially bounded functions on  $I.$

Define the weak-$L_{1}$ space $L_{1,\infty}(I)$ as a subset of  $S(I)$ equipped with the quasinorm defined by
\begin{equation}\label{weak-L1}
	\|f\|_{L_{1,\infty}(I)}=\sup\limits_{t>0}t\cdot \mu(t,f)=\sup\limits_{s>0} s\cdot m\{|f|>s\}, 
\end{equation}
where $\mu(f)$ is the decreasing rearrangement of the function $|f|.$ This space has the Fatou norm property, that is, if $\{f_n\}_{n\in \mathbb N}\subset L_{1,\infty}(I)$ such that $\|f_n\|_{1,\infty}\leq1$ for any $n\in\mathbb N,$ and $f_n\to f$ in measure as $n\to\infty,$ then $f\in L_{1,\infty}(I)$ and $\|f\|_{1,\infty}\leq1.$

For more details on these spaces, we refer the reader to \cite{G2008}.

The Lorentz space $\Lambda_{\log}(I)$ (see \cite{BSh}, \cite{KPS},  \cite{STZ_JFA})) is defined by setting
\begin{equation}\label{Lorentz-log}\Lambda_{\log}(I):=\left\{f\in S(I): \int_{0}^{m(I)}\frac{\mu(s,f)}{1+s}ds<\infty\right\}
	\end{equation}
equipped with the norm
$$\|f\|_{\Lambda_{\log}(I)}:=\int_{0}^{m(I)}\frac{\mu(s,f)}{1+s}ds.$$

Let $d\in\mathbb{N}.$ We shall need the following concave, increasing, and continuous function 
$\varphi_d$, given by the formula
\begin{equation}
	\varphi_{d}(t) =
	\begin{cases}
		\dfrac{1}{\Gamma(d+1)} 
		\displaystyle\int_{0}^{t} \log^{d}\!\left(\frac{1}{s}\right)\, ds + t, 
		& t \in (0,1), \\[12pt]
		2 + \dfrac{1}{\Gamma(d)} 
		\displaystyle\int_{1}^{t} 
		\left( \int_{0}^{1/s} \log^{d-1}\!\left(\frac{1}{u}\right)\, du \right) ds, 
		& t\in[1,\infty),
	\end{cases}
	\end{equation}
where $\Gamma(\cdot)$ is the Gamma function. 

The Lorentz space $\Lambda_{\varphi_d}(I)$ associated with the function $\varphi_d$ 
  is defined by setting
\begin{equation}\label{main-Lorentz}\Lambda_{\varphi_d}(I):=\left\{f\in S(I): \int_{0}^{m(I)}\mu(s,f)d\varphi_d(s)<\infty\right\}
		\end{equation}
equipped with the norm
$$\|f\|_{\Lambda_{\varphi_d}(I)}:=\int_{0}^{m(I)}\mu(s,f)d\varphi_d(s).$$

\subsection{Calder\'{o}n type operators and the Hilbert transform}
Throughout this paper we denote $\mathbb R_+:=(0,\infty).$ Define operators $C:(L_{1}+L_{\infty})(\mathbb R_+)\rightarrow (L_{1,\infty}+L_{\infty})(\mathbb R_+)$ by
\begin{equation}\label{Ces}(Cf)(t):=\frac{1}{t}\int_{0}^{t}f(s)ds , \,\ f\in(L_{1}+L_{\infty})(\mathbb R_+)
\end{equation}
and $C^{\ast}:\Lambda_{\log}(\mathbb R_+)\rightarrow S(\mathbb R_+) $ by
\begin{equation}\label{Cesaro-adjoint}
	(C^{\ast}f)(t):=\int_t^{\infty}f(s)\frac{ds}{s}, \,\, f\in\Lambda_{\log}(\mathbb R_+),
\end{equation}
where $C$ is called the Ces\`{a}ro operator \cite{STZ_JFA} (or else Hardy-Littlewood operator or Hardy operator as in \cite[Chapter II.3]{BSh},\cite[Chapter II.6]{KPS}) and $C^*$ its atjoint in $L_2$ sense.
For each $f\in \Lambda_{\log}(\mathbb R_+),$ define the Calder\'{o}n operator $S:\Lambda_{\log}(\mathbb R_+)\rightarrow(L_{1,\infty}+L_{\infty})(\mathbb R_+)$ as a sum of $C$ and $C^{\ast}$ by
\begin{equation}\label{S}
	(Sf)(t):=\frac1t\int_0^tf(s)ds+\int_t^{\infty}f(s)\frac{ds}{s}=(Cf)(t)+(C^{\ast}f)(t), \,\ f\in\Lambda_{\log}(\mathbb R_+).
\end{equation}
If $f\in \Lambda_{\log}(\mathbb{R}),$ then the classical Hilbert transform $H$ is defined by the principal-value integral
\begin{equation}\label{hilbert tr}
	(H f)(t)\overset{ \mathrm{p.v.} }{=}\frac{1}{\pi}\int_{\mathbb{R}}\frac{f(s)}{t-s}ds, \,\  f\in \Lambda_{\log}(\mathbb{R}),
\end{equation}
(see, e.g. \cite[Chapter III. 4]{BSh}).
\subsection{Integer powers of Calder\'{o}n type operators}

We denote by $\Lambda_{\log^d}(I)$ the Lorentz space defined using the function
$$
\psi(t) = \log^d(e^{d+1} + t), \quad t \geq 0.
$$
In this paper, we also need the notion of $d$ powers (or compositions) of the Ces\`{a}ro operators $C, C^*,$ and $S.$ For each $d\in\mathbb{N}$, let
$
C^d : (L_1 + L_\infty)(\mathbb R_+) \to S(\mathbb R_+)
$
be given by
\begin{equation}\label{Cesaro d power} (C^d f)(t)=\frac1{t\cdot\Gamma(d)}\int_0^tf(s)\log^{d-1}(\frac{t}{s})ds,\,\ f\in (L_1 + L_\infty)(\mathbb R_+).
\end{equation}
Its adjoint (in terms of the Hilbert spaces) operator
$
C^{* d} : \Lambda_{\log^d}(\mathbb R_+) \to S(\mathbb R_+)
$
is given by
\begin{equation}\label{Ces adj d power}(C^{\ast d} f)(t)=\frac1{\Gamma(d)}\int_t^{\infty}f(s)\log^{d-1}(\frac{s}{t})\frac{ds}{s}, \,\ f\in \Lambda_{\log^d}(\mathbb R_+),
\end{equation}
where $\Gamma(d):=(d-1)!$ is the Gamma function.
The Calder\'{o}n operator $
S^d : \Lambda_{\varphi_d}(\mathbb R_+) \to S(\mathbb R_+)
$
\begin{equation}\label{S d power}(S^d f)(t)=\int_{\mathbb R_+}\cdots\int_{\mathbb R_+} f(\eta)\min\{\frac{1}{t},\frac{1}{s}\}\cdots\min\{\frac{1}{\xi},\frac{1}{\eta}\}d\eta, \,\ f\in \Lambda_{\varphi_d}(\mathbb R_+).
\end{equation}
It follows from \cite[Lemma 3.6]{SYZZ} that the space $\Lambda_{\varphi_d}(\mathbb R_+)$ is a natural domain for the operator $S^d.$  In particular, for every $f\in \Lambda_{\varphi_d}(\mathbb R_+)$ we have $S^d(f)\in (L_1 +L_{\infty})(\mathbb R_+).$ For more information on these operators we refer the reader to \cite{SYZZ}.

\subsection{Tensor powers of Calder\'{o}n type operators and the multiple Hilbert transform}
Let $d\in\mathbb{N}$ and $\mathbb{R}_{+}^d=\{(t_1,\dots, t_d )\in\mathbb{R}^d: t_j>0,\,\ j=1,\dots d\}.$ 
For $f\in \Lambda_{\varphi_d}(\mathbb{R}_{+}^d)$ define $d$-tensor powers of the operators $C$, $C^{*},$ and $S$ by the following formulas
\begin{equation}\label{Ces-d-tensor}(C^{\otimes d}f)(t_1,\dots,t_d)
	=\frac{1}{t_1\cdots t_d}\int_{0}^{t_1}\cdots\int_{0}^{t_d}f(s_1,\dots,s_d)ds_1\cdots ds_d,
\end{equation}
\begin{equation}\label{Ces-adjoint-d-tensor}(C^{* \otimes d}f)(t_1,\dots,t_d)
	=\int_{t_1}^{\infty}\cdots\int_{t_d}^{\infty}f(s_1,\dots,s_d)\frac{ds_1}{s_1}\cdots \frac{ds_d}{s_d},
\end{equation}
and
\begin{equation}\label{S-d-tensor}
	(S^{\otimes d}f)(t_1,\dots,t_d):=\int_{\mathbb R_+}\cdots\int_{\mathbb R_+}f(s_1,\dots, s_d)\prod_{k=1}^d\min\{t_k^{-1},s_k^{-1}\}ds_1\cdots ds_d.
\end{equation}
Then for $s=(s_1,\dots,s_d),t=(t_1,\dots,t_d)\in\mathbb{R}_{+}^d$ kernels of these operators is defined respectively by 
\begin{equation}\label{Ces-d-tensor-kernel}
	K_{C^{\otimes d}}(s,t)=\frac{1}{t_1\cdots t_d}\chi_{\{s_1\leq t_1\}}\cdots\chi_{\{s_d\leq t_d\}},
\end{equation}
\begin{equation}\label{Ces-adjoint-d-tensor-kernel}
	K_{C^{*\otimes d}}(s,t)=\frac{1}{s_1\cdots s_d}\chi_{\{s_1> t_1\}}\cdots \chi_{\{s_d> t_d\}},
\end{equation}
and
\begin{equation}\label{S-d-tensor-kernel}
	K_{S^{\otimes d}}(s,t)=\prod_{k=1}^d\min\{t_k^{-1},s_k^{-1}\}.
\end{equation}
Similarly, for a function  $f\in \Lambda_{\varphi_d}(\mathbb{R}^d)$ define the multiple  Hilbert transform  by the principal-value integral
\begin{equation}\label{d-tensor-hilbert tr}
	(H^{\otimes d}f)(t_1,\dots,t_d)\overset{ \mathrm{p.v.} }{=}\frac{1}{\pi^d}\int_{\mathbb R}\cdots\int_{\mathbb R}\frac{f(s_1,\dots,s_d)}{\prod_{k=1}^d(t_k-s_k)}ds_1\cdots ds_d.
\end{equation}

\begin{rmk}
Our main result, Theorem \ref{main thm} shows that the space $\Lambda_{\varphi_d}(I)$ is the (natural) domain for the multiple Hilbert transform \eqref{d-tensor-hilbert tr} and $S^d$ \eqref{S d power}, therefore,  it will be used  frequently as the notation ${\rm dom}(S^d)$ i.e., ${\rm dom}(S^d):=\Lambda_{\varphi_d}(I).$
\end{rmk}

\subsection{Dilation on functions and sets}
For fixed values $t_1,\dots,t_d>0$, the $d$-dimensional anisotropic dilation is defined by
$$
\bigl(\sigma_{t_1,\dots,t_d} f\bigr)(s_1,\dots,s_d)
= f\!\left(\frac{s_1}{t_1},\dots,\frac{s_d}{t_d}\right).
$$
Let $f:\mathbb{R}_{+}^d\to\mathbb{R}_{+}$ be a measurable function. For a fixed $v>0$ set
$$\Omega_f(v)=f^{-1}(0,v),\ p_d(s)=s_1\cdots s_d, \,\ s=(s_1,\dots,s_d)\in \mathbb{R}^d_{+},$$
$$\sigma_{t_1,\dots,t_d}\Omega_f(v)=\Omega_{\sigma_{t_1,\dots,t_d}f}(v).$$
We have that
\begin{equation}\label{dilation-set-1,,,d}
	\sigma_{t_1,\dots,t_d} \Omega_{p_d}(v)=\Omega_{\sigma_{t_1,\dots,t_d}p_d}(v)=\Omega_{p_d}(t_1\cdots t_d\cdot v), \, v>0.
\end{equation}    
If $t_1=\dots=t_d=N\in\mathbb{N},$ we denote $\sigma_{t_1,\dots,t_d}$ simply by $\sigma_N,$ and in this setting we have
\begin{equation}\label{dilation-set-N}\sigma_N \Omega_{p_d}(v)=\Omega_{\sigma_{N}p_d}(v)=\Omega_{p_d}(N^d v),\,v>0.
\end{equation}
For $\xi=(\xi_1,\dots,\xi_d)\in\mathbb{R}^d$ let us also define the following sets
$$A_{d,\xi}(v):=\Omega_{p_d}(v)\cap (0,\xi_1)\times\dots\times(0,\xi_d),\ A_{d,1}(v)=A_{d,(1,\dots,1)},\ 0<v<1,
$$ where
\begin{equation}\label{A-set}A_{d,1}(v):=\Big\{(t_1,\dots,t_d)\in(0,1)^d:\ t_1t_2\cdots t_d<v\Big\},\, 0<v<1,
\end{equation}
and
$$B_{d,{\xi}}(v):=\Omega_{p_d}(v)\cap (\xi_1,+\infty)\times\dots\times(\xi_d,+\infty),\ B_{d,1}(v)=B_{d,(1,\dots,1)},\, v>1,
$$ that is,
\begin{equation}\label{B-set} B_{d,1}(v)=\Big\{(t_1,\dots,t_d)\in(1,\infty)^d:\ t_1t_2\cdots t_d<v\Big\},\, v>1.
\end{equation}
Hence, we have the following relations
\begin{equation}\label{A-set-dilation}\sigma_{t_1,\dots,t_d}\Big(A_{d,1}\Big(\frac{v}{t_1\cdots t_d}\Big)\Big)=A_{d,(t_1,\dots,t_d)}(v),\,0<v<1,
\end{equation}
and 
\begin{equation}\label{B-set-dilation}\sigma_{t_1,\dots,t_d}\Big(B_{d,1}\Big(\frac{v}{t_1\cdots t_d}\Big)\Big)=B_{d,(t_1,\dots,t_d)}(v),\, v>1,
\end{equation} respectively.
Throughout this paper, the symbol $\sim$ denotes asymptotic equivalence. That is, for functions $f$ and $g,$ we write
$f(t) \sim g(t), \,\ t \to t_0$ if $$\lim_{t \to t_0} \frac{f(t)}{g(t)} = 1.$$

\section{Tensor powers of Ces\`{a}ro operator and its adjoint}

The following result gives the lower distributional estimates for the $d$-densor product of the Ces\`{a}ro and its adjoint operators, which plays the key role in the proof of Theorem \ref{main thm} (ii).
\begin{thm}\label{c tensor power thm} Let $d\in\mathbb{N}$ and let $C^{\otimes d},$ $C^{\ast \otimes d}$ be the operators defined by \eqref{Ces-d-tensor} and \eqref{Ces-adjoint-d-tensor}, respectively. Let $f\in S(\mathbb R^d)$ be such that $\mu(f) \in{\rm dom}(C^d)\cap{\rm dom}(C^{*d}).$ There exists a sequence $\{f_{N,d}\}_{N\in\mathbb{N}}$ from $\mathbb{R}_+^d$ into $\mathbb{R}_+$ such that $\mu(f_{N,d})=\mu(f)\chi_{(0,N^d)}$ and
$$\liminf_{N\to\infty}\mu(C^{\otimes d}f_{N,d})\geq C^d\mu(f),\quad \liminf_{N\to\infty}\mu(C^{\ast \otimes d}f_{N,d})\geq C^{\ast d}\mu(f)$$
pointwise.
\end{thm}

For $d\in \mathbb{N},$ define
\begin{equation}\label{psi_d-function}\psi_d(t)=t\sum_{k=0}^{d-1}\frac{\log^k\big(\frac1t\big)}{k!},\quad t>0.
\end{equation}

The following lemma establishes several elementary but useful properties of $\psi_d.$
\begin{lem}\label{primative-lemma}Let $d\in \mathbb{N}.$
For $t>0,$
$$
\psi_d(t)=t\sum_{k=0}^{d-1}\frac{\log^k\big(\frac{1}{t}\big)}{k!}
$$
is a primitive of the function
$$
t\mapsto\frac{\log^{d-1}(\frac{1}{t})}{(d-1)!},
$$
and satisfies $\psi_d(1)=1$.

Moreover, $\psi_d$ is strictly increasing on $(0,1]$ and strictly decreasing on $[1,\infty)$ for even $d\geq 2,$ and $\psi_d$ is strictly increasing on $(0,\infty)$ for odd $d.$ 
\end{lem}
\begin{proof} Differentiating, we obtain
\begin{eqnarray*}\psi_d'(t)&=&\sum_{k=0}^{d-1}\frac{\log^k\big(\frac{1}{t}\big)}{k!}
	+ t\cdot\sum_{k=1}^{d-1}\frac{k\log^{k-1}(\frac1t)}{k!}\cdot (-\frac1t)\\
&=&\sum_{k=0}^{d-1}\frac{\log^k\big(\frac{1}{t}\big)}{k!}-\sum_{k=1}^{d-1}\frac{\log^{k-1}(\frac1t)}{(k-1)!}=\sum_{k=0}^{d-1}\frac{\log^k\big(\frac{1}{t}\big)}{k!}-\sum_{k=0}^{d-2}\frac{\log^k\big(\frac{1}{t}\big)}{k!}\\
&=&\frac{\log^{d-1}(\frac{1}{t})}{(d-1)!}.
	\end{eqnarray*}
	$$
	\psi_d(1)=1+\frac{\log\big(\frac{1}{1}\big)}{1!}+\cdots +\frac{\log^{d-1}\big(\frac{1}{1}\big)}{(d-1)!}=1.
	$$
	The case $d=1$ is immediate since
	$\psi_1(t)=t$ and hence,  $\psi_1'(t)=1.$  Assume \(d\ge2\). Since
	$$
	\psi_d'(t)=\frac{\log^{\,d-1}\big(\frac{1}{t}\big)}{(d-1)!},
	$$
	the sign assertions follow immediately from this formula and the fact that
	$\log(1/t)>0$ for \(0<t<1\), $\log(1/t)=0$ at $t=1, $  and $\log(1/t)<0$
	for $t>1,$ which completes the proof.
\end{proof}

We now establish a recurrence formula connecting $\psi_d$ and $\psi_{d-1}$, which plays a key role in the proof of subsequent results.
\begin{lem}\label{psid induction} Let $\psi_d$ be the function defined by \eqref{psi_d-function}. For $d\in \mathbb{N}$ and $\psi_0(v):=0$ we have
	$$ \psi_d(v)=v+\int_v^1\psi_{d-1}(\frac{v}{t})dt,\quad v>0.$$
	
\end{lem}
\begin{proof} Substituting $t=vs,$ we write
$$v+\int_v^1\psi_{d-1}(\frac{v}{t})dt=v+v\int_1^{v^{-1}}\psi_{d-1}(\frac1s)ds=v+v\int_1^{v^{-1}}\Big(\sum_{k=0}^{d-2}\frac{\log^k(s)}{k!}\Big)\frac{ds}{s}.$$	
	Substituting $s=e^u,$ we write
\begin{eqnarray*}v+\int_v^1\psi_{d-1}(\frac{v}{t})dt&=&v+v\int_0^{\log(\frac1v)}\Big(\sum_{k=0}^{d-2}\frac{u^k}{k!}\Big)du\\
&=&v+v\cdot\Big(\sum_{k=0}^{d-2}\frac{u^{k+1}}{(k+1)!}\Big|_{u=0}^{\log(\frac1v)}\Big)=v+v\cdot\sum_{k=0}^{d-2}\frac{\log^{k+1}(\frac1v)}{(k+1)!}\\
&=&v\Big(1+\sum_{k=1}^{d-1}\frac{\log^{k}(\frac1v)}{k!}\Big)=v\cdot\sum_{k=0}^{d-1}\frac{\log^{k}(\frac1v)}{k!}=\psi_d(v).
	\end{eqnarray*}
\end{proof}

The following lemma establishes a direct connection between the function $\psi_d$ and the measure of the set $A_{d,1}(v),\,\ 0<v<1$ defined in \eqref{A-set}.
\begin{lem}\label{first volume computation} Let $d\in \mathbb{N}.$ If $0<v<1$ and $A_{d,1}(v)$ be as in \eqref{A-set}, then
	$$m\Big(A_{d,1}(v)\Big)=\psi_d(v),$$
	where $\psi_d$ is the function defined by \eqref{psi_d-function}.
\end{lem}
\begin{proof} Let us prove the assertion by induction on $d.$ Base of induction (the case $d=1$) is obvious. Let us establish the step of induction. Assume the statement holds for $d-1$ and let us prove it for $d.$
		If $t_d\in(0,v),$ then for every $t_1,\dots,t_{d-1}\in(0,1)^{d-1}$ we have $t_1\cdots t_d<v.$ Therefore,
	$$A_{d,1}(v)=\Big((0,1)^{d-1}\times(0,v)\Big)\bigcup \Big\{(t_1,\dots,t_d): t_d\in(v,1),\ (t_1,\dots,t_{d-1})\in A_{d-1,1}(\frac{v}{t_d})\Big\}.$$
	
	Thus,
	$$m(A_{d,1}(v))=v+\int_v^1m(A_{d-1,1}(\frac{v}{t_d}))dt_d.$$
	
	By the induction hypothesis
	$$m(A_{d-1,1}(u))=\psi_{d-1}(u),\quad 0<u<1.$$
	Thus,
	$$m(A_{d,1}(v))=v+\int_v^1\psi_{d-1}(\frac{v}{t_d})dt_d.$$
	The step of induction follows now from Lemma \ref{psid induction}. This completes the proof.
\end{proof}

The following lemma gives the volume formula corresponding to the case $v>1.$
\begin{lem}\label{second volume computation}Let $d\in \mathbb{N}.$ If $v>1$ and $B_{d,1}(v)$ be as in \eqref{B-set}, then
	$$m\Big(B_{d,1}(v)\Big)=(-1)^{d-1}(\psi_d(v)-1),$$
	where $\psi_d$ is the function defined by \eqref{psi_d-function}.
\end{lem}
\begin{proof} We prove the assertion by induction on $d.$ Base of induction (the case $d=1$) is obvious. Let us establish the step of induction. Assume the statement holds for $d-1$ and let us prove it for $d.$
	Clearly,
	$$B_{d,1}(v)=\Big\{(t_1,\dots,t_d):\ t_d\in(1,v),\ (t_1,\dots,t_{d-1})\in B_{d-1,1}(\frac{v}{t_d})\Big\}.$$
	Thus,
	$$m(B_{d,1}(v))=\int_1^vm(B_{d-1,1}(\frac{v}{t_d}))dt_d.$$
	
	By the induction hypothesis
	$$m(B_{d-1,1}(u))=(-1)^{d-2}(\psi_{d-1}(u)-1).$$
	Thus,
	\begin{eqnarray*}m(B_{d,1}(v))&=&(-1)^{d-2}\int_1^v(\psi_{d-1}(\frac{v}{t_d})-1)dt_d\\
		&=&(-1)^{d-1}\int_v^1(\psi_{d-1}(\frac{v}{t_d})-1)dt_d\\
	&=&(-1)^{d-1}\Big(v-1+\int_v^1\psi_{d-1}(\frac{v}{t_d})dt_d\Big)\stackrel{L.\ref{psid induction}}{=}(-1)^{d-1}(\psi_d(v)-1).
	\end{eqnarray*}
	This yields the step of induction and, hence, completes the proof.
\end{proof}
Let $d\in\mathbb{N}.$ For any fixed $N\in\mathbb{N}$ we set
\begin{equation}\label{psi_N,d-function}
\psi_{N,d}(t)=N^d\psi_d(\frac{t}{N^d}),\quad t>0.
\end{equation}
Differentiating as in Lemma \ref{primative-lemma} gives
$$\psi_{N,d}'(t)=\psi_d'(\frac{t}{N^d})=\frac{\log^{d-1}\big(\frac{N^d}{t}\big)}{(d-1)!},\quad t>0.$$
For $t\in(0,N^d)$ we have $\log(\frac{N^d}{t})>0$, hence, $\psi_{N,d}'(t)>0$. 
Thus, $\psi_{N,d}$ is strictly increasing on $(0,N^d).$ 
Consequently, it is bijective on $(0,N^d).$

We now demonstrate that a mapping induced by $\psi_{N,d}$ is measure preserving, a property that will be crucial for the subsequent tensor estimates.

\begin{lem}\label{gamma_is_measure_preserving_d}Let $\psi_{N,d}$ be the function defined by \eqref{psi_N,d-function}. For every $d,N\in\mathbb{N},$ the map $\gamma_{N,d}:(0,N)^d\to(0,N^d)$ given by the formula
$$\gamma_{N,d}(t_1,\dots,t_d)=\psi_{N,d}(\prod_{k=1}^dt_k),\quad t=(t_1,\dots ,t_d),$$
is measure preserving.
\end{lem}
\begin{proof} For every $a\in(0,N^d),$ we have
\begin{eqnarray*}\gamma_{N,d}^{-1}((0,a))&=&\Big\{(t_1,\dots,t_d)\in(0,N)^d:\ \psi_{N,d}(\prod_{k=1}^dt_k)<a\Big\}\\
	&=&\Big\{(t_1,\dots,t_d)\in(0,N)^d:\ \prod_{k=1}^dt_k<\psi_{N,d}^{-1}(a)\Big\}\\
	&=&\Big\{(t_1,\dots,t_d)\in(0,N)^d:\ \prod_{k=1}^dt_k<N^d\psi_d^{-1}(\frac{a}{N^d})\Big\}\\
	&=&\sigma_N\Big(\Big\{(t_1,\dots,t_d)\in(0,1)^d:\ \prod_{k=1}^dt_k<\psi_d^{-1}(\frac{a}{N^d})\Big\}\Big).
\end{eqnarray*}
Since
$$m(\sigma_N(A))=N^dm(A),\quad A\subset\mathbb{R}^d,$$
it follows that
$$m(\gamma_{N,d}^{-1}((0,a)))=N^d\cdot m\Big(\Big\{(t_1,\dots,t_d)\in(0,1)^d:\ \prod_{k=1}^dt_k<\psi_d^{-1}(\frac{a}{N^d})\Big\}\Big).$$
Using Lemma \ref{first volume computation}, we write
$$m(\gamma_{N,d}^{-1}((0,a)))=N^d\cdot\psi_d(\psi_d^{-1}(\frac{a}{N^d})) =N^d\cdot\frac{a}{N^d}=a.$$
This suffices to conclude that $\gamma_{N,d}:(0,N)^d\to(0,N^d)$ is measure preserving.
\end{proof}

Let $d,N\in\mathbb{N}.$ For given $f\in S(I),$ set
\begin{equation}\label{def-f_N}f_{N,d}(t_1,\cdots,t_d)=
	\begin{cases}
		\mu(\psi_{N,d}(t_1\cdots t_d),f),& t_1,\dots,t_d\in(0,N),\\
		0,& \max(t_1,\dots,t_d)\geq N.
	\end{cases}
\end{equation}

As a preparatory step, we establish the relation between $f_{N,d}$
and the original function $f$ through their distribution functions.

\begin{lem}\label{d-mu fn compute lemma} Let $d\in \mathbb{N}$ and let $f_{N,d}$ be the function defined by \eqref{def-f_N}.  We have
	$$\mu(f_{N,d})=\mu(f)\chi_{(0,N^d)}.$$
\end{lem}
\begin{proof} On $(0,N)^d$ we have
$$f_{N,d}=\mu(f)\circ\gamma_{N,d}=(\mu(f)\chi_{(0,N^d)})\circ\gamma_{N,d}.$$
As $\gamma_{N,d}:(0,N)^d\to(0,N^d),$ it follows that
$$\mu(f_{N,d}|_{(0,N)^d})=\mu((\mu(f)\chi_{(0,N^d)})\circ\gamma_{N,d})=\mu((\mu(f)\chi_{(0,N^d)}))=\mu(f)\chi_{(0,N^d)}.$$
Since $f_{N,d}=0$ outside $(0,N)^d,$ the assertion follows.
\end{proof}

The following lemma provides a reduction of the $d$-fold integral of a function depending on the product of variables to a one-dimensional integral via a suitable change of variables and scaling.
\begin{lem}\label{pre first main equality lemma} For any measurable function $F:\mathbb{R}_+\to\mathbb{R}_+$ such that the formula below makes sense we have
$$\int_0^{t_1}\cdots\int_0^{t_d}F(\prod_{k=1}^ds_k)\prod_{k=1}^dds_k=t\int_0^1(\sigma_{\frac1t}F\circ\psi_d^{-1})(u)du,\quad t=\prod_{k=1}^dt_k,$$
where $\psi_d$ is the function defined by \eqref{psi_d-function}. 
\end{lem}
\begin{proof} Substitute $s_k=t_ku_k$ for $1\leq k\leq d.$ We have
$$F(\prod_{k=1}^ds_k)=F(\prod_{k=1}^dt_k\cdot\prod_{k=1}^du_k)=F(t\cdot\prod_{k=1}^du_k)=(\sigma_{\frac1t}F)(\prod_{k=1}^du_k),$$
$$\prod_{k=1}^dds_k=\prod_{k=1}^dt_k\cdot\prod_{k=1}^ddu_k=t\cdot\prod_{k=1}^ddu_k.$$
Thus,
$$
\begin{aligned}
	&\int_0^{t_1}\cdots\int_0^{t_d}
	F\left(\prod_{k=1}^d s_k\right)
	\prod_{k=1}^d ds_k
	\\
	&\qquad
	=\left(\prod_{k=1}^d t_k\right)
	\int_0^1\cdots\int_0^1
	F\left(\prod_{k=1}^d t_k u_k\right)
	\prod_{k=1}^d du_k
	\\
	&\qquad
	=t\int_0^1\cdots\int_0^1
	(\sigma_{1/t}F)\left(\prod_{k=1}^d u_k\right)
	\prod_{k=1}^d du_k,
	\qquad
	t=\prod_{k=1}^d t_k.
\end{aligned}
$$
Since
$$
\psi_d(\prod_{k=1}^d u_k)
=\gamma_{1,d}(u_1,\ldots,u_d),
$$
and $\gamma_{1,d}:(0,1)^d\to(0,1)$ is measure-preserving by Lemma \ref{gamma_is_measure_preserving_d}, 
we may write
$$
\begin{aligned}
	 &\int_0^1\cdots\int_0^1
	(\sigma_{1/t}F)\left(\prod_{k=1}^d u_k\right)
	\prod_{k=1}^ddu_k
	\\
	&=
	\int_0^1\cdots\int_0^1
	(\sigma_{1/t}F\circ\psi_d^{-1})
	\left(\gamma_{1,d}(u_1,\ldots,u_d)\right)
	\prod_{k=1}^ddu_k.
	\\
	&=
	\int_0^1
	(\sigma_{1/t}F\circ\psi_d^{-1})(u)\,du.
\end{aligned}
$$
Consequently,
$$
	\int_0^{t_1}\cdots\int_0^{t_d}
	F\left(\prod_{k=1}^d s_k\right)
	\prod_{k=1}^d ds_k
	=
	t\int_0^1
	(\sigma_{1/t}F\circ\psi_d^{-1})(u)\,du,
	\quad
	t=\prod_{k=1}^d t_k.
$$
This completes the proof.
\end{proof}

The following lemma computes 
$C^{\otimes d}f_{N,d}$ explicitly, allowing us to reduce the problem to one-dimensional distribution functions.

\begin{lem}\label{first main equality lemma} Let $d\in \mathbb{N}$  and let $C^{\otimes d}$ be the operator defined by \eqref{Ces-d-tensor}. If $t_1,\dots,t_d\in(0,N),$ then
$$\big(C^{\otimes d}f_{N,d}\big)(t_1,\dots,t_d)=\frac{1}{t_1\cdots t_d}\int_{0}^{t_1\cdots t_d}\mu\big(\psi_{N,d}(v),f\big)\,\frac{\log^{d-1}\big(\frac{t_1\cdots t_d}{v}\big)}{(d-1)!}dv,$$
where $f_{N,d}$ and $\psi_{N,d}$ are the functions defined by \eqref{def-f_N} and \eqref{psi_N,d-function}, respectively.
\end{lem}
\begin{proof} Denote for brevity $t=\prod\limits_{k=1}^dt_k.$ If $t_1,\dots,t_d\in(0,N),$ then by \eqref{Ces-d-tensor} and \eqref{def-f_N} we have
$$(C^{\otimes d}f_{N,d})(t_1,\dots,t_d)=\frac1{t_1\cdots t_d}\int_{0}^{t_1}\cdots\int_{0}^{t_d}\mu(\psi_{N,d}(\prod_{k=1}^ds_k),f)\prod_{k=1}^dds_k.$$
Denote for brevity
$$F_{N,d}:u\to \mu(\psi_{N,d}(u),f),\quad u\in(0,N^d).$$
We have
$$(C^{\otimes d}f_{N,d})(t_1,\dots,t_d)=\frac1{t_1\cdots t_d}\int_{0}^{t_1}\cdots\int_{0}^{t_d}F_{N,d}(\prod_{k=1}^ds_k)\prod_{k=1}^dds_k.$$
Applying Lemma \ref{pre first main equality lemma} to the function $F_{N,d},$ we obtain
$$(C^{\otimes d}f_{N,d})(t_1,\dots,t_d)= \int_0^1(\sigma_{\frac1t}F_{N,d}\circ\psi_d^{-1})(u)du.$$
Substituting $u=\psi_d(w),$ we write
\begin{eqnarray*}(C^{\otimes d}f_{N,d})(t_1,\dots,t_d)&=& \int_0^1(\sigma_{\frac1t}F_{N,d})(w)\psi_d'(w)dw\\
	&=&\frac1{(d-1)!}\int_0^1F_{N,d}(wt)\log^{d-1}(\frac1w)dw\\
	&=&\frac1{(d-1)!t}\int_0^tF_{N,d}(v)\log^{d-1}(\frac{t}{v})dv.
\end{eqnarray*}
\end{proof}

We can now compute the decreasing rearrangement of $C^{\otimes d}f_{N,d}.$

	\begin{lem}\label{lem:C-rearrangement}
		Let $d,N\in\mathbb N$, and let $C^{\otimes d}$ be the operator
		defined by \eqref{Ces-d-tensor}. For every $t\in(0,N^d)$, we have
		\begin{align*}
			\mu\bigl(t,C^{\otimes d}f_{N,d}\bigr)
			&\ge
			\mu\bigl(t,
			\chi_{(0,N)^d}C^{\otimes d}f_{N,d}
			\bigr)                                                     \\
			&=
			\frac{1}{\psi_{N,d}^{-1}(t)}
			\int_0^{\psi_{N,d}^{-1}(t)}
			\mu\bigl(\psi_{N,d}(v),f\bigr)
			\frac{
				\log^{d-1}\!\left(
				\frac{\psi_{N,d}^{-1}(t)}{v}
				\right)
			}{(d-1)!}\,dv .
		\end{align*}
	\end{lem}
	\begin{proof}
		Set
		$$
		G_{N,d}(u)
		:=
		\frac1u
		\int_0^u
		\mu\bigl(\psi_{N,d}(v),f\bigr)
		\frac{\log^{d-1}(u/v)}{(d-1)!}\,dv,
		\quad 0<u<N^d.
		$$
		By Lemma~11, for every
		$(t_1,\ldots,t_d)\in(0,N)^d$,
		$$
		(C^{\otimes d}f_{N,d})(t_1,\ldots,t_d)
		=
		G_{N,d}(t_1\cdots t_d).
		$$
		The function $G_{N,d}$ is decreasing. Since
		$$
		\gamma_{N,d}(t_1,\ldots,t_d)
		=
		\psi_{N,d}(t_1\cdots t_d)
		$$
		is measure-preserving from $(0,N)^d$ onto $(0,N^d)$ by Lemma \ref{gamma_is_measure_preserving_d}, the
		functions
		$$
		\chi_{(0,N)^d}C^{\otimes d}f_{N,d}
		$$
		and
		$$
		u\to
		G_{N,d}\bigl(\psi_{N,d}^{-1}(u)\bigr),
		\quad u\in(0,N^d),
		$$
		are equimeasurable. Moreover, since both $G_{N,d}$ and $\psi_{N,d}^{-1}$ are,
		respectively, decreasing and increasing, the function
		$$
		u\to
		G_{N,d}\bigl(\psi_{N,d}^{-1}(u)\bigr)
		$$
		is decreasing. Therefore, for every $t\in(0,N^d)$,
		\begin{align*}
			&\mu\bigl(
			t,\chi_{(0,N)^d}C^{\otimes d}f_{N,d}
			\bigr)
			=
			G_{N,d}\bigl(\psi_{N,d}^{-1}(t)\bigr)\\ =&
			\frac{1}{\psi_{N,d}^{-1}(t)}
			\int_0^{\psi_{N,d}^{-1}(t)}
			\mu\bigl(\psi_{N,d}(v),f\bigr)
			\frac{
				\log^{d-1}\!\left(
				\frac{\psi_{N,d}^{-1}(t)}{v}
				\right)
			}{(d-1)!}\,dv.
		\end{align*}
		Finally,
		$$
		0\le
		\chi_{(0,N)^d}C^{\otimes d}f_{N,d}
		\le C^{\otimes d}f_{N,d},
		$$
		and hence,
		$$
		\mu\bigl(
		t,C^{\otimes d}f_{N,d}
		\bigr)
		\ge
		\mu\bigl(
		t,\chi_{(0,N)^d}C^{\otimes d}f_{N,d}
		\bigr).
		$$
		This completes the proof.
	\end{proof}

\begin{lem}\label{second main equality lemma2} If $t_1,\dots,t_d\in(0,N],$ then
$$\int_{t_1}^{\infty}\cdots\int_{t_d}^{\infty}f_{N,d}(s_1,\cdots,s_d)\frac{ds_1\cdots ds_d}{s_1\cdots s_{d}}=\int_{t_1\cdots t_d}^{N^d}\frac{\mu(\psi_{N,d}(v),f)}{v}\frac{\log^{d-1}\big(\frac{v}{t_1\cdots t_{d}}\big)}{(d-1)!}dv,$$
{\color{blue}where $f_{N,d}$ and $\psi_{N,d}$ are the functions defined by  \eqref{def-f_N} and \eqref{psi_N,d-function}, respectively.}
\end{lem}
\begin{proof} Denote for brevity
$$g_{N,d}(s):=\frac1s\mu\big(\psi_{N,d}(s),f\big)\,\chi_{(0,N^d)}(s),\quad s>0.$$
For $t=(t_1,\dots,t_d)\in\mathbb{R}^d_+$ and $v\in\mathbb{R}_+,$ recall the notation
$$B_{d,t}(v)=\Big\{(s_1,\dots,s_d):\ s_k\geq t_k,\quad 1\leq k\leq d,\quad s_1\cdots s_d\leq v\Big\}.$$
We have
$$\int_{t_1}^{\infty}\cdots\int_{t_d}^{\infty}f_{N,d}(s_1,\cdots,s_d)\frac{ds_1\cdots ds_d}{s_1\cdots s_{d}}=\int_{t_1\cdots t_d}^{\infty}g_{N,d}(v)dm(B_{d,t}(v)),$$
where $m(B_{d,t}(v))$ is the Lebesgue measure of the set $B_{d,t}(v).$
Indeed, observe that the integrand can be written as a function of the single variable
$$
v=s_1\cdots s_d:
\quad
\mu\bigl(\psi_{N,d}(s_1\cdots s_d),f\bigr)
\frac{ds_1\cdots ds_d}{s_1\cdots s_d}
=
g_{N,d}(s_1\cdots s_d)\,ds_1\cdots ds_d.
$$

Therefore, by the push-forward change of variables formula for the map
$$
(s_1,\ldots,s_d)\mapsto s_1\cdots s_d,
$$
we obtain
\begin{eqnarray*}
	& &\int_{t_1}^{\infty}\cdots\int_{t_d}^{\infty}
	\mu\bigl(\psi_{N,d}(s_1\cdots s_d),f\bigr)
	\frac{ds_1\cdots ds_d}{s_1\cdots s_d}
	\\
	&=&
	\int_{\mathbb R_+^d}
	g_{N,d}(s_1\cdots s_d)
	\prod_{k=1}^d\chi_{[t_k,\infty)}(s_k)
	\,ds_1\cdots ds_d
	\\
	&
	=&
	\int_{t_1\cdots t_d}^{\infty}
	g_{N,d}(v)\,
	d\bigl(m(B_{d,t}(v))\bigr).
\end{eqnarray*}
Here, the lower limit is \(t_1\cdots t_d\), since
$$
s_k\geq t_k
\quad (1\leq k\leq d)
\quad\Longrightarrow\quad
s_1\cdots s_d\geq t_1\cdots t_d,
$$
and  $g_{N,d}$ is nonnegative measurable function.
This proves the claimed equality.
On the other hand, we have
\begin{eqnarray*}
m(B_{d,t}(v)) &\stackrel{\eqref{B-set-dilation}}{= }&p_d(t) \cdot m\Big(B_{d,1}(\frac{v}{p_d(t)})\Big)\\
&\stackrel{L.\ref{second volume computation}}{=}&(-1)^{d-1}t_1t_2\cdots t_d\times\Big(\psi_d(\frac{v}{t_1t_2\cdots t_d})-1\Big),
\end{eqnarray*}
where $p_d(t)=t_1t_2\cdots t_d.$
Hence,
$$\frac{d}{dv}m(B_{d,t}(v))=(-1)^{d-1}\psi_d'(\frac{v}{t_1t_2\cdots t_d})\stackrel{L.\ref{primative-lemma}}{=}\frac{\log^{d-1}(\frac{v}{t_1t_2\cdots t_d})}{(d-1)!}.$$
It follows that
$$\int_{t_1}^{\infty}\cdots\int_{t_d}^{\infty}\mu(\psi_{N,d}(s_1\cdots s_d),f)\frac{ds_1\cdots ds_d}{s_1\cdots s_{d}}$$
$$=\frac1{(d-1)!}\int_{t_1\cdots t_d}^{\infty}g_{N,d}(v)\log^{d-1}(\frac{v}{t_1t_2\cdots t_d})dv.$$
The assertion follows now from the definition of $g_{N,d}.$
\end{proof}

The following lemma computes $(C^{\ast})^{\otimes d}f_{N,d}$ explicitly, allowing us to reduce the problem to one-dimensional setting, which is the key relation to prove Theorem \ref{c tensor power thm}.
\begin{lem}\label{second main equality lemma} Let $(C^{\ast})^{\otimes d}$ be the operator defined by \eqref{Ces-adjoint-d-tensor}. If $t_1,\dots,t_d\in(0,N),$ then
$$\big((C^{\ast})^{\otimes d}f_{N,d}\big)(t_1,\dots,t_d)$$
$$=\sum_{S\subsetneq\{1,\dots,d\}}
(-1)^{|S|}\int_{N^{|S|}\prod_{i\notin S}t_i}^{N^d}\frac{\mu(\psi_{N,d}(v),f)}{v}\frac{\log^{d-1}\big(\frac{v}{N^{|S|}\prod_{i\notin S}t_i}\big)}{(d-1)!}dv,$$
where $f_{N,d}$ and $\psi_{N,d}$ are the functions defined by \eqref{def-f_N} and \eqref{psi_N,d-function}, respectively.
\end{lem}
\begin{proof} By definition,
$$\big((C^{\ast})^{\otimes d}f_{N,d}\big)(t_1,\dots,t_d)=\int_{t_1}^N\cdots\int_{t_d}^Nf_{N,d}(s_1,\dots,s_d)\frac{ds}{s_1\cdots s_d}.$$
Using inclusion-exclusion principle, we write
$$\int_{t_1}^N\cdots\int_{t_d}^Nf_{N,d}(s_1,\dots,s_d)\frac{ds_1\cdots ds_d}{s_1\cdots s_d}=$$
$$=\sum\limits_{S\subseteq\{1,\cdots,d\}}
(-1)^{|S|}
\int\limits_{\prod_{i=1}^d A_i(S)}
f_{N,d}(s_1,\cdots,s_d)\frac{ds}{s_1\cdots s_d},$$
where
$$A_i(S)=
\begin{cases}
[N,\infty), & i\in S,\\
[t_i,\infty), & i\notin S.
\end{cases}.$$
Since the summand corresponding to $S=\{1,\dots,d\}$ is zero because $f_{N,d}$ vanishes on $(N^d,\infty)$ (see \eqref{def-f_N}), it follows that
$$
\begin{aligned}
	&\int_{t_1}^N\cdots\int_{t_d}^N
	f_{N,d}(s_1,\ldots,s_d)
	\frac{ds_1\cdots ds_d}{s_1\cdots s_d}
	\\
	&\quad=
	\sum_{S\subsetneq\{1,\ldots,d\}}
	(-1)^{|S|}
	\int_{\prod_{i=1}^d A_i(S)}
	f_{N,d}(s_1,\ldots,s_d)
	\frac{ds_1\cdots ds_d}{s_1\cdots s_d}.
\end{aligned}
$$
By Lemma \ref{second main equality lemma2}, we have
$$\int\limits_{\prod_{i=1}^d A_i(S)}
f_{N,d}(s_1,\dots,s_d)\frac{ds_1\cdots ds_d}{s_1\cdots s_d}=$$
$$=\int_{N^{|S|}\prod_{i\notin S}t_i}^{N^d}\frac{\mu(\psi_{N,d}(v),f)}{v}\frac{\log^{d-1}\big(\frac{v}{N^{|S|}\prod_{i\notin S}t_i}\big)}{(d-1)!}dv$$
for any $S\subsetneq\{1,\ldots,d\}.$ Combining this with the previous one, we complete the proof.
\end{proof}

\begin{lem}\label{main mu Cast lemma}Let $d\in\mathbb N.$ For $S\subsetneq\{1,\dots,d\}$ and for every $(t_1,\dots,t_d)\in(0,N)^d $ consider the function 
$$F_{N,S}:(t_1,\dots,t_d)\to \int_{N^{|S|}\prod_{i\notin S}t_i}^{N^d}\frac{\mu(\psi_{N,d}(v),f)}{v}\frac{\log^{d-1}\big(\frac{v}{N^{|S|}\prod_{i\notin S}t_i}\big)}{(d-1)!}dv.$$
 We have
$$\mu(t,F_{N,S})=\int_{N^{|S|}\psi_{N,d-|S|}^{-1}(\frac{t}{N^{|S|}})}^{N^d}\frac{\mu(\psi_{N,d}(v),f)}{v}\frac{\log^{d-1}\big(\frac{v}{N^{|S|}\psi_{N,d-|S|}^{-1}(\frac{t}{N^{|S|}})}\big)}{(d-1)!}dv$$
for $t\in(0,N^d).$
\end{lem}
\begin{proof}
	Put
		$$
	r=|S|,\quad m=d-|S|.
	$$
	Since $S\subsetneq\{1,\ldots,d\},$ we have $m\geq1.$ Define
		$$
	H_r(u)
	:=
	\int_u^{N^d}
	\frac{\mu(\psi_{N,d}(v),f)}{v}
	\frac{\log^{d-1}(v/u)}{(d-1)!}\,dv,
	\quad 0<u<N^d.
	$$
	Then, by the definition of $F_{N,S},$
		$$
	F_{N,S}(t_1,\ldots,t_d)
	=
	H_r\left(
	N^r\prod_{i\notin S}t_i
	\right).
	$$
It can be shown that $H_r$ is decreasing.  Now define
		$$
	P_S(t_1,\ldots,t_d)
	=
	N^r\prod_{i\notin S}t_i.
	$$
	Then,	
	$$
	F_{N,S}=H_r\circ P_S.
	$$
We next determine the distribution of $P_S.$ Since $P_S$ does not depend on the coordinates $t_i,$ $i\in S,$ for $0<u<N^m$ we have
		$$
	\begin{aligned}
		&m\left(
		\left\{
		(t_1,\ldots,t_d)\in(0,N)^d:
		P_S(t_1,\ldots,t_d)\leq N^r u
		\right\}
		\right)
		\\
		&\quad
		=
		N^r
		m\left(
		\left\{
		(t_i)_{i\notin S}\in(0,N)^m:
		\prod_{i\notin S}t_i\leq u
		\right\}
		\right).
	\end{aligned}
	$$
	After the change of variables
	$$
	t_i=Nx_i,\qquad 0<x_i<1,
	$$
	we obtain
		$$
	\begin{aligned}
		&m\left(
		\left\{
		(t_i)_{i\notin S}\in(0,N)^m:
		\prod_{i\notin S}t_i\leq u
		\right\}
		\right)
		\\
		&\qquad
		=
		N^m
		m\left(
		\left\{
		(x_i)_{i\notin S}\in(0,1)^m:
		\prod_{i\notin S}x_i\leq\frac{u}{N^m}
		\right\}
		\right).
	\end{aligned}
	$$
	Since $	\gamma_{1,m}$ is measure-preserving by Lemma \ref{gamma_is_measure_preserving_d}, we have
		$$
	m\left(
	\left\{
	x\in(0,1)^m:
	\gamma_{1,m}(x)\leq y
	\right\}
	\right)
	=y,
	\qquad 0<y<1.
	$$
	Therefore,
		$$
	m\left(
	\left\{
	(t_i)_{i\notin S}\in(0,N)^m:
	\prod_{i\notin S}t_i\leq u
	\right\}
	\right)
	=u.
	$$
	Consequently,
		$$
		m\left(
		\left\{
		t\in(0,N)^d:
		P_S(t)\leq N^r u
		\right\}
		\right)
		=N^r u.
	$$
	Thus, \(P_S\) has the same distribution as the identity function on \((0,N^d)\).
	More precisely, the measure-preserving map
	$$
	\gamma_{1,m}:(0,1)^m\rightarrow(0,1)
	$$
	implies that
	$$
	\psi_{N,m}^{-1}:(0,N^m)\rightarrow(0,N)^m
	$$
	gives the corresponding quantile representation of the product variable. Hence,
	$$
	P_S
	\quad\text{and}\quad
	N^r\psi_{N,m}^{-1}\left(\frac{t}{N^r}\right)
	$$	have the same distribution on their respective measure spaces.
	Since \(H_r\) is decreasing, applying the same decreasing function \(H_r\) to two equimeasurable functions preserves equimeasurability. Therefore,
	$$
	H_r(P_S)
	\quad\text{and}\quad
	H_r\left(
	N^r\psi_{N,m}^{-1}\left(\frac{t}{N^r}\right)
	\right)
	$$
	are equimeasurable.
	But,	
	$$
	H_r(P_S(t_1,\ldots,t_d))
	=
	F_{N,S}(t_1,\ldots,t_d),
	$$
	while
		$$
H_r\left(
		N^r\psi_{N,m}^{-1}\left(\frac{t}{N^r}\right)
		\right)
	=
		\int_{N^r\psi_{N,m}^{-1}(t/N^r)}^{N^d}
		\frac{\mu(\psi_{N,d}(v),f)}{v}
		\frac{
			\log^{d-1}
			\left(
			\frac{v}
			{N^r\psi_{N,m}^{-1}(t/N^r)}
			\right)
		}{(d-1)!}\,dv.
	$$
	Hence, the latter function is precisely the decreasing rearrangement of \(F_{N,S}\). Therefore,
		$$
			\mu(t,F_{N,S})
			=
			\int_{N^{|S|}
				\psi_{N,d-|S|}^{-1}
				\left(\frac{t}{N^{|S|}}\right)}^{N^d}
			\frac{\mu(\psi_{N,d}(v),f)}{v}
		\cdot
			\frac{
				\log^{d-1}
				\left(
				\frac{v}
				{N^{|S|}
					\psi_{N,d-|S|}^{-1}
					\left(\frac{t}{N^{|S|}}\right)}
				\right)
			}{(d-1)!}\,dv,
		$$
		for $0<t<N^d.$
	This proves the lemma.
\end{proof}

We first establish a simple but useful inequality for the inverse of the function $\psi_{N,d},$ which will be used in later estimates.
\begin{lem}\label{first trivial lemma} Let $d\in \mathbb{N}$ and let $\psi_{N,d}$ be the function defined by \eqref{psi_N,d-function}. We have
	$$\log(\frac{\psi_{N,d}^{-1}(t)}{\psi_{N,d}^{-1}(w)})\geq\log(\frac{t}{w}),\quad 0<w<t<N^d.$$	
\end{lem}
\begin{proof} It is immediate that
	$$1\geq \frac{\sum_{k=0}^{d-1}\frac{\log^k\big(\frac{N^d}{a}\big)}{k!}}{\sum_{k=0}^{d-1}\frac{\log^k\big(\frac{N^d}{b}\big)}{k!}},\quad 0<b<a<N^d.$$
	Multiplying by $\frac{a}{b},$ we write
	$$\frac{a}{b}\geq \frac{\psi_{N,d}(a)}{\psi_{N,d}(b)},\quad 0<b<a<N^d.$$
	Substituting $a=\psi_{N,d}^{-1}(t)$ and $b=\psi_{N,d}^{-1}(w),$ we write
	$$\frac{\psi_{N,d}^{-1}(t)}{\psi_{N,d}^{-1}(w)}\geq \frac{t}{w},\quad 0<w<t<N^d.$$
	Taking logarithms, we complete the proof.
\end{proof}

We next record an asymptotic property of the function 
$\psi^{-1}_{d},$ which will be crucial in the limit computations for large $N.$
\begin{lem}\label{second trivial lemma} Let $d\in \mathbb{N}$ and let $\psi_{d}$ be the function defined by \eqref{psi_d-function}. For every $t>0,$ we have
	$$\lim_{N\to\infty}\frac{\log\big(\frac1{\psi_d^{-1}(\frac{t}{N})}\big)}{\log(N)}=1.$$
\end{lem}
\begin{proof} Taking into account that
	as $t \to 0$, the largest power of the logarithm dominates in
	$$
	\psi_d(t) = t \sum_{k=0}^{d-1} \frac{\log^k(\frac{1}{t})}{k!} \sim t \frac{\log^{d-1}(\frac{1}{t})}{(d-1)!},\quad t \to 0.
	$$
	Formally inverting the leading term gives
	$$
	s = \psi_d(t) \sim t \frac{\log^{d-1}(\frac{1}{t})}{(d-1)!}, \quad t \to 0,$$
	which implies
	$$t\sim \frac{(d-1)! \, s}{\log^{d-1}(\frac{1}{t})}, \quad t \to 0.
	$$
	Finally, since $t \to 0$ implies $s \to 0$, we may replace $\log(\frac{1}{t})$ by $\log(\frac{1}{s})$ up to $1+o(1)$ terms, yielding
	$$
	\psi_d^{-1}(s) =\frac{(d-1)! \, s}{\log^{d-1}(\frac{1}{s})} \cdot (1+o(1)), \quad s \to 0.
	$$ Putting $s=\frac{t}{N},$
	we write
	$$\frac1{\psi_d^{-1}(\frac{t}{N})}=\frac{N\cdot\log^{d-1}(\frac{N}{t})}{(d-1)!t}\cdot (1+o(1)),\quad N\to\infty.$$
	Thus,
	$$\log\big(\frac1{\psi_d^{-1}(\frac{t}{N})}\big)=\log\big(N\cdot\log^{d-1}(\frac{N}{t})\big)-\log((d-1)!t)+o(1),\quad N\to\infty.$$
	Therefore,
	$$
	\lim_{N \to \infty} \frac{\log\Big(\frac{1}{\psi_d^{-1}(\frac{t}{N})}\Big)}{\log(N)} = 1,
	$$
	which gives the desired result.
\end{proof}

The following lemma provides precise asymptotics for 
$\psi_{N,d}^{-1}$ and its derivative, a key step in evaluating the limiting behavior of our main expressions.
\begin{lem}\label{third trivial lemma} Let $d\in \mathbb{N}$ and let $\psi_{N,d}$ be the function defined by \eqref{psi_N,d-function}. For every $t>0,$ we have
	$$\lim_{N\to\infty}(\log^{d-1}(N)\cdot \psi_{N,d}^{-1})'(t)=\frac{(d-1)!}{d^{d-1}},\quad \lim_{N\to\infty}\log^{d-1}(N)\cdot\psi_{N,d}^{-1}(t)=\frac{(d-1)!}{d^{d-1}}t.$$
\end{lem}
\begin{proof} Recall that $\psi_{N,d}(t)=N^d\psi_d(\frac{t}{N^d}),$ $0<t<N^d.$ Hence, $\psi_{N,d}'(t)=\psi_d'(\frac{t}{N^d}),$ $0<t<N^d.$ Also, we have that $N^{-d}\psi_{N,d}^{-1}(t)=\psi_d^{-1}(\frac{t}{N^d}),$ $0<t<N^d.$ Using the Chain Rule, we write
	$$(\psi_{N,d}^{-1})'(t)=\frac1{\psi_{N,d}'(\psi_{N,d}^{-1}(t))}=\frac1{\psi_d'(N^{-d}\psi_{N,d}^{-1}(t))}=\frac{(d-1)!}{\log^{d-1}\big(\frac1{\psi_d^{-1}(\frac{t}{N^d})}\big)}.$$
	The first assertion follows now from Lemma \ref{second trivial lemma}.	
	
	Next, we compute the second limit. Taking into account that
	$$\psi_d^{-1}(s)=\frac{(d-1)!s}{\log^{d-1}(\frac{1}{s})}\cdot(1+o(1)),\quad s\to0,$$
	we write
	$$\lim_{N\to\infty}\log^{d-1}(N)\cdot\psi_{N,d}^{-1}(t)=\lim_{N\to\infty}\log^{d-1}(N)\cdot\frac{(d-1)!\cdot\frac{t}{N^d}}{\log^{d-1}(\frac{N^d}{t})}=\frac{(d-1)!}{d^{d-1}}t.$$
	This completes the computation of the second limit.
\end{proof}

As a preparatory step, we show an upper estimate for the derivative of the scaled $\psi_{N,d}^{-1}.$
\begin{lem}\label{fourth trivial lemma} Let $d\in \mathbb{N}.$ If $\psi_{N,d}$ and $\psi_d$ are the functions defined by \eqref{psi_N,d-function} and \eqref{psi_d-function}, respectively, then we have
	$$(\log^{d-1}(N)\cdot\psi_{N,d}^{-1})'(t)\leq 2^{d-1}\cdot (d-1)!,\quad 0<t<\psi_d(N^{-\frac12})N^d.$$
\end{lem}
\begin{proof} Recall that $\psi_{N,d}(t)=N^d\psi_d(\frac{t}{N^d}),$ $0<t<N^d.$ Hence, $\psi_{N,d}'(t)=\psi_d'(\frac{t}{N^d}),$ $0<t<N^d.$ Also, we have, $N^{-d}\psi_{N,d}^{-1}(t)=\psi_d^{-1}(\frac{t}{N^d}),$ $0<t<N^d.$ Using the Chain Rule, we write
	$$(\psi_{N,d}^{-1})'(t)=\frac1{\psi_{N,d}'(\psi_{N,d}^{-1}(t))}=\frac1{\psi_d'(N^{-d}\psi_{N,d}^{-1}(t))}=\frac{(d-1)!}{\log^{d-1}\big(\frac1{\psi_d^{-1}(\frac{t}{N^d})}\big)}.$$
	Therefore,
	$$(\log^{d-1}(N)\cdot\psi_{N,d}^{-1})'(t)=\log^{d-1}(N)\cdot(\psi_{N,d}^{-1})'(t)=\frac{(d-1)!\cdot \log^{d-1}(N)}{\log^{d-1}\big(\frac1{\psi_d^{-1}(\frac{t}{N^d})}\big)}.$$
	Hence, the claimed inequality is the same as
	$$\frac{(d-1)!\cdot \log^{d-1}(N)}{\log^{d-1}\big(\frac1{\psi_d^{-1}(\frac{t}{N^d})}\big)}\leq 2^{d-1}\cdot (d-1)!,\quad 0<t<\psi_d(N^{-\frac12})N^d,$$
	which is the same as
	$$\frac{1}{2^{d-1}}\log^{d-1}(N)\leq \log^{d-1}\big(\frac1{\psi_d^{-1}(\frac{t}{N^d})}\big),\quad 0<t<\psi_d(N^{-\frac12})N^d.$$
	Taking $\frac{1}{d-1}$ power from the both sides, we obtain
	$$\frac12\log(N)\leq\log\big(\frac1{\psi_d^{-1}(\frac{t}{N^d})}\big)\quad 0<t<\psi_d(N^{-\frac12})N^d.$$
	Taking exponentials from the both sides we get that the previous one is the same as
	$$\psi_d^{-1}(\frac{t}{N^d})\leq N^{-\frac12},\quad 0<t<\psi_d(N^{-\frac12})N^d,$$
	Acting with $\psi_d,$ we obtain
	$0<t<\psi_d(N^{-\frac12})N^d,$
	thereby completing the proof.
\end{proof}

\begin{lem}\label{first convergence lemma} Let $d\in \mathbb{N}$ and let $C^{ d}$ be the operator defined by \eqref{Cesaro d power}.  For every $t>0,$ we have
$$\liminf_{N\to\infty}\frac1{\psi_{N,d}^{-1}(t)}\int_0^{\psi_{N,d}^{-1}(t)}\mu(\psi_{N,d}(v),f)\frac{\log^{d-1}(\frac{\psi_{N,d}^{-1}(t)}{v})}{(d-1)!}dv\geq (C^d\mu(f))(t).$$
Here, $\psi_{N,d}$ is the function defined by \eqref{psi_N,d-function}.
\end{lem}
\begin{proof} We consider only those $N$ with $t<\psi_d(N^{-\frac12})N^d.$ 
By Lemma \ref{third trivial lemma}, 
$$\mu(w,f)\frac{\log^{d-1}(\frac{t}{w})}{(d-1)!}(\log^{d-1}(N)\cdot\psi_{N,d}^{-1})'(w)\to d^{1-d}\mu(w,f)\log^{d-1}(\frac{t}{w}),\quad N\to\infty.$$
By Lemma\footnote{It is exactly at this place we are using the assumption $t<\psi_d(N^{-\frac12})N^d.$} \ref{fourth trivial lemma},
$$\mu(w,f)\frac{\log^{d-1}(\frac{t}{w})}{(d-1)!}(\log^{d-1}(N)\cdot\psi_{N,d}^{-1})'(w)\leq 2^{d-1}\mu(w,f)\log^{d-1}(\frac{t}{w}).$$
Using Dominated Convergence Theorem, we infer
$$\lim_{N\to\infty}\int_0^t\mu(w,f)\frac{\log^{d-1}(\frac{t}{w})}{(d-1)!}(\log^{d-1}(N)\cdot\psi_{N,d}^{-1})'(w)dw$$
$$=d^{1-d}\int_0^t\mu(w,f)\log^{d-1}(\frac{t}{w})dw.$$
Again using Lemma \ref{third trivial lemma}, we conclude that
$$\lim_{N\to\infty}\frac1{\log^{d-1}(N)\cdot\psi_{N,d}^{-1}(t)}\int_0^t\mu(w,f)\frac{\log^{d-1}(\frac{t}{w})}{(d-1)!}(\log^{d-1}(N)\cdot\psi_{N,d}^{-1})'(w)dw$$
$$\stackrel{L. \ref{third trivial lemma}}{=}\frac1{t\cdot(d-1)!}\int_0^t\mu(w,f)\log^{d-1}(\frac{t}{w})dw\stackrel{\eqref{Cesaro d power}}{=}(C^d\mu(f))(t).$$

Substituting $v=\psi_{N,d}^{-1}(w),$ we write
\begin{eqnarray*}&&\frac1{\psi_{N,d}^{-1}(t)}\int_0^{\psi_{N,d}^{-1}(t)}\mu(\psi_{N,d}(v),f)\frac{\log^{d-1}(\frac{\psi_{N,d}^{-1}(t)}{v})}{(d-1)!}dv\\
&=&\frac1{\psi_{N,d}^{-1}(t)}\int_0^t\mu(w,f)\frac{\log^{d-1}(\frac{\psi_{N,d}^{-1}(t)}{\psi_{N,d}^{-1}(w)})}{(d-1)!}(\psi_{N,d}^{-1})'(w)dw \\
&\stackrel{L.\ref{first trivial lemma}}{\geq}& \frac1{\psi_{N,d}^{-1}(t)}\int_0^t\mu(w,f)\frac{\log^{d-1}(\frac{t}{w})}{(d-1)!}(\psi_{N,d}^{-1})'(w)dw\\
&=&\frac1{\log^{d-1}(N)\cdot\psi_{N,d}^{-1}(t)}\int_0^t\mu(w,f)\frac{\log^{d-1}(\frac{t}{w})}{(d-1)!}(\log^{d-1}(N)\cdot\psi_{N,d}^{-1})'(w)dw.
\end{eqnarray*}
Combining the last two equations, we complete the proof.
\end{proof}

To facilitate later estimates involving integrals of $\psi_{N,d}^{-1},$ we first bound its logarithmic derivative.
\begin{lem}\label{fifth trivial lemma}  Let $d\in \mathbb{N}$ and let $\psi_{N,d}$ be the function defined by \eqref{psi_N,d-function}. There exists $c_d\in(0,1)$ such that
	$$\frac{(\psi_{N,d}^{-1})'(w)}{\psi_{N,d}^{-1}(w)}\leq\frac2w,\quad 0<w<c_dN^d.$$	
\end{lem}
\begin{proof}
Let $d=1.$ Since
	$$
	\psi_{N,1}(t)=t,
	$$
	one has
	$$
	\frac{(\psi_{N,1}^{-1})'(w)}
	{\psi_{N,1}^{-1}(w)}
	=\frac1w\le\frac2w,
	$$
	so the statement is still true. Now let $d\geq 2.$ We claim that there  there exists $T>0$ such that for every $t\geq T$ we have
	$$
	\frac{t^{\,d-1}}{(d-1)!}\;\geq\;\frac12\sum_{k=0}^{d-1}\frac{t^{\,k}}{k!}.
	$$
	Indeed, dividing both sides by $\frac{t^{d-1}}{(d-1)!},$ we have
	$$
	1+\sum_{k=0}^{d-2}\frac{(d-1)!}{k!}\,t^{\,k-(d-1)} \;\leq\; 2.
	$$
	Define
	$$
	S(t):=\sum_{k=0}^{d-2}\frac{(d-1)!}{k!}\,t^{\,k-(d-1)}.
	$$
	For every $k\leq d-2,$ we have that $k-(d-1)\leq -1$, hence for $t\geq1$ we obtain
	$t^{\,k-(d-1)}\leq t^{-1}$. Thus for $t\geq1$
	$$
	S(t)\leq \frac{1}{t}\sum_{k=0}^{d-2}\frac{(d-1)!}{k!}.
	$$
	Choose $T:=\sum\limits_{k=0}^{d-2}\frac{(d-1)!}{k!} $. Then for all $t\geq T$,
	$$
	S(t)\leq\frac{\sum\limits_{k=0}^{d-2}\frac{(d-1)!}{k!}}{T}= 1,
	$$
	so $1+S(t)\leq2$ and the claimed inequality holds for every $t\geq T$. Let $t=\log \tfrac{1}{u}>0$ and set
	$$
	u_0:=e^{-T}\in(0,1).
	$$
	Define
	$$
	c_d:=\psi_d(u_0)\in(0,1).
	$$
	Since $\psi_d$ is continuous, strictly increasing, and satisfies $\psi_d(0+)=0$ and $\psi_d(1)=1,$ $\psi_d^{-1}(c_d)=u_0$. Therefore, there is $c_d\in(0,1)$ such that for every $0<u<\psi_d^{-1}(c_d)$
	(equivalently, every $t>T$) we have
	$$
	\frac{\log^{d-1}(\frac1u)}{(d-1)!}\;\geq\;\frac12\sum_{k=0}^{d-1}\frac{\log^k(\frac1u)}{k!}.
	$$
	By Lemma \ref{primative-lemma} we know 
	$$\psi_d(u)=u\sum_{k=0}^{d-1}\frac{\log^k(\frac1u)}{k!}\quad \text{and}\quad \psi_d'(u)=\frac{\log^{d-1}(\frac1u)}{(d-1)!}.$$
	Therefore, the previous inequality is equivalent to
	$$\psi_d'(u)\cdot u\geq\frac{\psi_d(u)}{2},\quad 0<u<\psi_d^{-1}(c_d).$$
	Denoting $u=\psi_d^{-1}(s),$ this is the same as
	$$\psi_d'(\psi_d^{-1}(s))\cdot \psi_d^{-1}(s)\geq\frac{s}{2},\quad 0<s<c_d.$$
	Using the Chain Rule, we write
	$$(\psi_d^{-1})'(s)=\frac1{\psi_d'(\psi_d^{-1}(s))}.$$
	Hence, the last inequality is the same as
	$$\frac{(\psi_d^{-1})'(s)}{\psi_d^{-1}(s)}\leq\frac{2}{s},\quad 0<s<c_d.$$
	Recall that $\psi_{N,d}(w)=N^d\psi_d(\frac{w}{N^d}),$ $0<w<N^d.$ Hence, $\psi_{N,d}^{-1}(w)=N^d\psi_d^{-1}(\frac{w}{N^d}),$ $0<w<N^d.$ Denoting $s=\frac{w}{N^d},$ the previous estimate can be written as
	$$\frac{(\psi_d^{-1})'(\frac{w}{N^d})}{\psi_d^{-1}(\frac{w}{N^d})}\leq\frac{2N^d}{w},\quad 0<w<c_dN^d.$$
	Since $(\psi_{N,d}^{-1})'(w)=(\psi_d^{-1})'(\frac{w}{N^d}),$ it follows that
	$$\frac{(\psi_{N,d}^{-1})'(w)}{\psi_{N,d}^{-1}(w)}=\frac{(\psi_d^{-1})'(\frac{w}{N^d})}{N^d\psi_{d}^{-1}(\frac{w}{N^d})}\leq\frac2w,\quad 0<w<c_dN^d,$$
	thereby completing the proof.
\end{proof}

\begin{lem}\label{second convergence lemma} Let $d\in \mathbb{N}$ and let $C^{* d}$ be the operator defined by \eqref{Ces adj d power}.  For every $t>0,$ we have
$$\liminf_{N\to\infty}\int_{\psi_{N,d}^{-1}(t)}^{N^d}\frac{\mu(\psi_{N,d}(v),f)}{v}\frac{\log^{d-1}\big(\frac{v}{\psi_{N,d}^{-1}(t)}\big)}{(d-1)!}dv\geq ((C^{\ast})^d\mu(f))(t).$$
Here, $\psi_{N,d}$ is the function defined by \eqref{psi_N,d-function}.
\end{lem}
\begin{proof} We only consider those $N$ with $t<c_dN^d$ (here, $c_d$ is the constant featuring in Lemma \ref{fifth trivial lemma}). By Lemma \ref{third trivial lemma},
$$\mu(w,f)\frac{\log^{d-1}\big(\frac{w}{t}\big)}{(d-1)!}\frac{(\psi_{N,d}^{-1})'(w)}{\psi_{N,d}^{-1}(w)}\to \mu(w,f)\frac{\log^{d-1}\big(\frac{w}{t}\big)}{(d-1)!w},\quad N\to\infty.$$
By Lemma \ref{fifth trivial lemma},
$$\mu(w,f)\frac{\log^{d-1}\big(\frac{w}{t}\big)}{(d-1)!}\frac{(\psi_{N,d}^{-1})'(w)}{\psi_{N,d}^{-1}(w)}\leq 2\mu(w,f)\frac{\log^{d-1}\big(\frac{w}{t}\big)}{(d-1)!w},\quad w<c_dN^d.$$
Using Dominated Convergence Theorem, we infer
$$\lim_{N\to\infty}\int_t^{c_dN^d}\mu(w,f)\frac{\log^{d-1}\big(\frac{w}{t}\big)}{(d-1)!}\cdot\frac{(\psi_{N,d}^{-1})'(w)}{\psi_{N,d}^{-1}(w)}dw$$
$$=\int_t^{\infty}\mu(w,f)\frac{\log^{d-1}\big(\frac{w}{t}\big)}{(d-1)!}\cdot\frac{1}{w}dw\stackrel{\eqref{Ces adj d power}}{=}((C^{\ast})^d\mu(f))(t).$$
Substituting $v=\psi_{N,d}^{-1}(w),$ we write
$$\int_{\psi_{N,d}^{-1}(t)}^{N^d}\frac{\mu(\psi_{N,d}(v),f)}{v}\frac{\log^{d-1}\big(\frac{v}{\psi_{N,d}^{-1}(t)}\big)}{(d-1)!}dv$$
$$=\int_t^{N^d}\mu(w,f)\frac{\log^{d-1}\big(\frac{\psi_{N,d}^{-1}(w)}{\psi_{N,d}^{-1}(t)}\big)}{(d-1)!}\frac{(\psi_{N,d}^{-1})'(w)}{\psi_{N,d}^{-1}(w)}dw$$
$$\stackrel{L.\ref{first trivial lemma}}{\geq}\int_t^{c_dN^d}\mu(w,f)\frac{\log^{d-1}\big(\frac{w}{t}\big)}{(d-1)!}\frac{(\psi_{N,d}^{-1})'(w)}{\psi_{N,d}^{-1}(w)}dw.$$
Combining the last two equations, we complete the proof.
\end{proof}

\begin{lem}\label{sixth trivial lemma} Let $d_1,d_2\geq1$ be such that $d_1+d_2=d.$ For every $t\in(0,1)$ and for every $N\in\mathbb N$ with $\log(N)\geq t^{-1},$ we have
$$\left|\log\big(\frac{\psi_{N,d}^{-1}(w)\log^{d_2-1}(N)}{t}\big)\right|^{d-1}\leq c_d\Big(\log^{d-1}(w)+\log^{d-1}(\frac{e}{t})\Big),\quad w\in(t\log^{d_1}(N),N^d).$$
\end{lem}
\begin{proof} We have
$$\left|\log\big(\frac{\psi_{N,d}^{-1}(w)\log^{d_2-1}(N)}{t}\big)\right|\leq|\log(\psi_{N,d}^{-1}(w))|+\log(\frac1t)+(d_2-1)\log(\log(N)).$$
Since $\psi_d^{-1}(u)\leq u, \,\ 0<u<1,$ 
it follows that for $0<w<N^d$ we have $0<\frac{w}{N^d}<1,$ and therefore,
$$\psi_{N,d}^{-1}(w)=N^d\psi_d^{-1}(\frac{w}{N^d})\leq N^d\cdot \frac{w}{N^d}=w$$
and
$$\psi_{N,d}^{-1}(w)=N^d\psi_d^{-1}(\frac{w}{N^d})\geq N^d\psi_d^{-1}(\frac{t\log^{d_1}(N)}{N^d})$$
$$\stackrel{\log(N)\geq\frac1t}{\geq} N^d\psi_d^{-1}(\frac{\log^{d_1-1}(N)}{N^d})\geq c_d^{(1)}\log^{-d_2}(N)$$
for $w\in(t\log^{d_1}(N),N^d),$ it follows that
$$\left|\log\big(\frac{\psi_{N,d}^{-1}(w)\log^{d_2-1}(N)}{t}\big)\right|\leq\log(w)+|\log(c_d^{(1)})|+\log(\frac1t)+(2d_2-1)\log(\log(N))$$
for $w\in(t\log^{d_1}(N),N^d).$ Next,
$$\log(\log(N))\leq\frac1{d_1}\log(\frac{w}{t})\leq\log(w)+\log(\frac1t)$$
for $w\in(t\log^{d_1}(N),N^d).$ Thus,
$$\left|\log\big(\frac{\psi_{N,d}^{-1}(w)\log^{d_2-1}(N)}{t}\big)\right|\leq 2d_2\log(w)+|\log(c_d^{(1)})|+2d_2\log(\frac1t)$$
for $w\in(t\log^{d_1}(N),N^d).$ Hence,
$$\left|\log\big(\frac{\psi_{N,d}^{-1}(w)\log^{d_2-1}(N)}{t}\big)\right|^{d-1}$$
$$\leq(6d_2)^{d-1}\Big(\log^{d-1}(w)+|\log(c_d^{(1)})|^{d-1}+\log^{d-1}(\frac1t)\Big)$$
for $w\in(t\log^{d_1}(N),N^d).$ 
Taking the $(d-1)$-st root, we get
$$
\begin{aligned}
	\left|
	\log\left(
	\frac{\psi_{N,d}^{-1}(w)\log^{d_2-1}(N)}{t}
	\right)
	\right|
	&\leq
	6d_2
	\left(
	\log^{d-1}(w)
	+
	|\log(c_d^{(1)})|^{d-1}
	+
	\log^{d-1}\left(\frac1t\right)
	\right)^{\frac1{d-1}}.
\end{aligned}
$$
By the elementary inequality
$$
(a+b+c)^{\frac1{d-1}}
\leq
C_d\left(
a^{\frac1{d-1}}
+
b^{\frac1{d-1}}
+
c^{\frac1{d-1}}
\right),
\quad a,b,c\geq0,
$$
we therefore have
$$
\begin{aligned}
	\left|
	\log\left(
	\frac{\psi_{N,d}^{-1}(w)\log^{d_2-1}(N)}{t}
	\right)
	\right|
	&\leq
	C_d
	\left(
	\log w
	+
	|\log(c_d^{(1)})|
	+
	\log\frac1t
	\right).
\end{aligned}
$$
Since $0<t<1,$
$$
\log\frac et=1+\log\frac1t\geq1,
$$
and, since $c_d^{(1)}$ depends only on $d,$
$$
|\log(c_d^{(1)})|
\leq
C_d\log\frac et.
$$
Moreover,
$$
\log\frac1t\leq\log\frac et.
$$
Consequently, after renaming the constant, we obtain
$$
\begin{aligned}
	\left|
	\log\left(
	\frac{\psi_{N,d}^{-1}(w)\log^{d_2-1}(N)}{t}
	\right)
	\right|
	&\leq
	c_d
	\left(
	\log w+\log\frac et
	\right)
\end{aligned}
$$
for every
$
w\in\left(t\log^{d_1}(N),N^d\right).
$
Raising both sides to the power $d-1$, and using
$$
(a+b)^{d-1}\le 2^{d-2}\bigl(a^{d-1}+b^{d-1}\bigr),
\quad a,b\ge0,
$$
we prove the assertion.
\end{proof}

\begin{lem}\label{third convergence lemma} Let $1\leq d_1,d_2\leq d$ be such that $d_1+d_2=d.$ Assume that
	\begin{equation}\label{log-tail-assumption}
		\int_1^\infty \mu(w,f)\bigl(1+\log^{d-1}w\bigr)\,\frac{dw}{w}<\infty.
	\end{equation} 
	For every $t>0,$ we have
$$\lim_{N\to\infty}\int_{N^{d_1}\psi_{N,d_2}^{-1}(\frac{t}{N^{d_1}})}^{N^d}\frac{\mu(\psi_{N,d}(v),f)}{v}\log^{d-1}\big(\frac{v}{N^{d_1}\psi_{N,d_2}^{-1}(\frac{t}{N^{d_1}})}\big)dv=0.$$
Here, $\psi_{N,d}$ is the function defined by \eqref{psi_N,d-function}.
\end{lem}
\begin{proof} 
	We split the proof into four steps.
	
	{\bf Step 1:} For every $t\in(0,1),$ we have
$$\lim_{N\to\infty}\int_{t\log^{d_1}(N)}^{N^d}\mu(w,f)\Big|\log\big(\frac{\psi_{N,d}^{-1}(w)\log^{d_2-1}(N)}{t}\big)\Big|^{d-1}\frac{dw}{w}=0.$$
Indeed, for every $N\in\mathbb{N}$ with $\log(N)\geq\frac1t,$ it follows from Lemma \ref{sixth trivial lemma} that
$$\int_{t\log^{d_1}(N)}^{N^d}\mu(w,f)\Big|\log\big(\frac{\psi_{N,d}^{-1}(w)\log^{d_2-1}(N)}{t}\big)\Big|^{d-1}\frac{dw}{w}$$
$$\leq c_d\int_{t\log^{d_1}(N)}^{N^d}\mu(w,f)\log^{d-1}(w)\frac{dw}{w}+c_d\log^{d-1}(\frac{e}{t})\int_{t\log^{d_1}(N)}^{N^d}\mu(w,f)\frac{dw}{w}$$
$$\leq c_d\int_{t\log^{d_1}(N)}^{\infty}\mu(w,f)\log^{d-1}(w)\frac{dw}{w}+c_d\log^{d-1}(\frac{e}{t})\int_{t\log^{d_1}(N)}^{\infty}\mu(w,f)\frac{dw}{w}.$$
Since $t\log^{d_1}N\to\infty$, both terms tend to zero by
\eqref{log-tail-assumption}.
This yields the claim in Step 1.

{\bf Step 2:} For every $t\in(0,1),$ we have
$$\lim_{N\to\infty}\int_{t\log^{d_1}(N)}^{N^d}\mu(w,f)\Big|\log\big(\frac{\psi_{N,d}^{-1}(w)\log^{d_2-1}(N)}{t}\big)\Big|^{d-1}\frac{(\psi_{N,d}^{-1})'(w)}{\psi_{N,d}^{-1}(w)}dw=0.$$
Let $\alpha_d$ be the constant in Lemma \ref{fifth trivial lemma}. It is asserted in  Lemma \ref{fifth trivial lemma} that
$$\frac{(\psi_{N,d}^{-1})'(w)}{\psi_{N,d}^{-1}(w)}\leq \frac2w,\quad 0<w<\alpha_dN^d.$$
Therefore,
\begin{eqnarray*}& &\int_{t\log^{d_1}(N)}^{N^d}\mu(w,f)\Big|\log\big(\frac{\psi_{N,d}^{-1}(w)\log^{d_2-1}(N)}{t}\big)\Big|^{d-1}\frac{(\psi_{N,d}^{-1})'(w)}{\psi_{N,d}^{-1}(w)}dw\\
&\leq& 2\int_{t\log^{d_1}(N)}^{\alpha_dN^d}\mu(w,f)\Big|\log\big(\frac{\psi_{N,d}^{-1}(w)\log^{d_2-1}(N)}{t}\big)\Big|^{d-1}\frac{dw}{w}\\
&+&\int_{\alpha_dN^d}^{N^d}\mu(w,f)\Big|\log\big(\frac{\psi_{N,d}^{-1}(w)\log^{d_2-1}(N)}{t}\big)\Big|^{d-1}\frac{(\psi_{N,d}^{-1})'(w)}{\psi_{N,d}^{-1}(w)}dw.
\end{eqnarray*}
Before estimating the last integral, note that
\begin{eqnarray*}\frac{\psi_{N,d}^{-1}(w)\log^{d_2-1}(N)}{t}&\geq&\psi_{N,d}^{-1}(w)\log^{d_2-1}(N)\\
&\geq&\psi_{N,d}^{-1}(\alpha_dN^d)\log^{d_2-1}(N)\\
&=&\psi_d^{-1}(\alpha_d)\cdot N^d\log^{d_2-1}(N)\\
&\geq&1,\quad w\in(\alpha_dN^d,N^d)
\end{eqnarray*}
for all sufficiently large $N.$ Hence, absolute values can be dropped in the last integral. We can now estimate it as follows
\begin{eqnarray*}& &\int_{\alpha_dN^d}^{N^d}\mu(w,f)\log^{d-1}\big(\frac{\psi_{N,d}^{-1}(w)\log^{d_2-1}(N)}{t}\big)\frac{(\psi_{N,d}^{-1})'(w)}{\psi_{N,d}^{-1}(w)}dw\\
&\leq& \mu(\alpha_dN^d,f)\cdot\int_{\alpha_dN^d}^{N^d}\log^{d-1}\big(\frac{\psi_{N,d}^{-1}(w)\log^{d_2-1}(N)}{t}\big)\frac{(\psi_{N,d}^{-1})'(w)}{\psi_{N,d}^{-1}(w)}dw\\
&\leq& \mu(\alpha_dN^d,f)\cdot\log^{d-1}\big(\frac{N^d\log^{d_2-1}(N)}{t}\big)\cdot\int_{\alpha_dN^d}^{N^d}\frac{(\psi_{N,d}^{-1})'(w)}{\psi_{N,d}^{-1}(w)}dw\\
&\stackrel{w=\psi_{N,d}(v)}{=}&\mu(\alpha_dN^d,f)\cdot\log^{d-1}\big(\frac{N^d\log^{d_2-1}(N)}{t}\big)\cdot\int_{\psi_d^{-1}(\alpha_d)\cdot N^d}^{N^d}\frac{dv}{v}\\
&=&\mu(\alpha_dN^d,f)\log^{d-1}\big(\frac{N^d\log^{d_2-1}(N)}{t}\big)\cdot\log(\frac1{\psi_d^{-1}(\alpha_d)}).
\end{eqnarray*}
We have
\begin{eqnarray*}\log(\frac{N^d\log^{d_2-1}(N)}{t})&\leq&\log(\frac{N^{d+d_2-1}}{t})\\
&\leq&(d+d_2-1)\big(\log(N)+\log(\frac{e}{t})\big)\\
&\leq &2d\big(\log(N)+\log(\frac{e}{t})\big).
\end{eqnarray*}
Hence,
\begin{eqnarray*}& &\int_{\alpha_dN^d}^{N^d}\mu(w,f)\log^{d-1}\big(\frac{\psi_{N,d}^{-1}(w)\log^{d_2-1}(N)}{t}\big)\frac{(\psi_{N,d}^{-1})'(w)}{\psi_{N,d}^{-1}(w)}dw\\
&\leq &(4d)^{d-1}\log(\frac1{\psi_d^{-1}(\alpha_d)})\cdot \mu(\alpha_dN^d,f)\cdot(\log^{d-1}(N)+\log^{d-1}(\frac{e}{t}))\\
&\leq &c_d^{(1)}\cdot \mu(\alpha_dN^d,f)\cdot(\log^{d-1}(\frac12\alpha_dN^d)+\log^{d-1}(\frac{e}{t})).
\end{eqnarray*}
Note that
\begin{eqnarray*}\mu(\alpha_dN^d,f)\cdot\log^{d-1}(\frac12\alpha_dN^d)&=&\frac1{\log(2)}\int_{\frac12\alpha_dN^d}^{\alpha_dN^d}\frac{dw}{w}\cdot \mu(\alpha_dN^d,f)\cdot\log^{d-1}(\frac12\alpha_dN^d)\\
&\leq&\frac1{\log(2)}\int_{\frac12\alpha_dN^d}^{\alpha_dN^d}\mu(w,f)\frac{dw}{w}\cdot\log^{d-1}(\frac12\alpha_dN^d)\\
&\leq&\frac1{\log(2)}\int_{\frac12\alpha_dN^d}^{\alpha_dN^d}\mu(w,f)\log^{d-1}(w)\frac{dw}{w}\\
&\leq &\frac1{\log(2)}\int_{\frac12\alpha_dN^d}^{\infty}\mu(w,f)\log^{d-1}(w)\frac{dw}{w}
\end{eqnarray*}
and
\begin{eqnarray*}\mu(\alpha_dN^d,f)&=&\frac1{\log(2)}\int_{\frac12\alpha_dN^d}^{\alpha_dN^d}\frac{dw}{w}\cdot \mu(\alpha_dN^d,f)\\
&\leq&\frac1{\log(2)}\int_{\frac12\alpha_dN^d}^{\alpha_dN^d}\mu(w,f)\frac{dw}{w}\\
&\leq &\frac1{\log(2)}\int_{\frac12\alpha_dN^d}^{\infty}\mu(w,f)\frac{dw}{w}.
\end{eqnarray*}
Hence,
\begin{eqnarray*}& &\int_{\alpha_dN^d}^{N^d}\mu(w,f)\log^{d-1}\big(\frac{\psi_{N,d}^{-1}(w)\log^{d_2-1}(N)}{t}\big)\frac{(\psi_{N,d}^{-1})'(w)}{\psi_{N,d}^{-1}(w)}dw\\
&\leq&\frac{c_d^{(1)}}{\log(2)}\cdot\int_{\frac12\alpha_dN^d}^{\infty}\mu(w,f)\log^{d-1}(w)\frac{dw}{w}\\
&+&\frac{c_d^{(1)}}{\log(2)}\cdot\log^{d-1}(\frac{e}{t})\cdot\int_{\frac12\alpha_dN^d}^{\infty}\mu(w,f)\frac{dw}{w}.
\end{eqnarray*}
Finally,
\begin{eqnarray*}& &\int_{t\log^{d_1}(N)}^{N^d}\mu(w,f)\Big|\log\big(\frac{\psi_{N,d}^{-1}(w)\log^{d_2-1}(N)}{t}\big)\Big|^{d-1}\frac{(\psi_{N,d}^{-1})'(w)}{\psi_{N,d}^{-1}(w)}dw\\
&\leq& 2\int_{t\log^{d_1}(N)}^{\alpha_dN^d}\mu(w,f)\Big|\log\big(\frac{\psi_{N,d}^{-1}(w)\log^{d_2-1}(N)}{t}\big)\Big|^{d-1}\frac{dw}{w}\\
&+&\frac{c_d^{(1)}}{\log(2)}\cdot\int_{\frac12\alpha_dN^d}^{\infty}\mu(w,f)\log^{d-1}(w)\frac{dw}{w}\\
&+&\frac{c_d^{(1)}}{\log(2)}\cdot\log^{d-1}(\frac{e}{t})\cdot\int_{\frac12\alpha_dN^d}^{\infty}\mu(w,f)\frac{dw}{w}.
\end{eqnarray*}
According to Step 1, the first summand on the right hand side tends to $0$ as $N\to\infty.$ The second and the third summands on the right hand side obviously tend to $0$ as $N\to\infty.$ This yields the claim in Step 2.

{\bf Step 3:} For every $t\in(0,1),$ we have
$$\lim_{N\to\infty}\int_{\frac{t}{\log^{d_2-1}(N)}}^{N^d}\mu(\psi_{N,d}(v),f)\log^{d-1}\big(\frac{v\log^{d_2-1}(N)}{t}\big)\frac{dv}{v}=0.$$
For every $t\in(0,1),$ we obtain
\begin{eqnarray*}\psi_{N,d}(\frac{t}{\log^{d_2-1}(N)})&=&N^d\psi_d(N^{-d}\frac{t}{\log^{d_2-1}(N)})\\
	&=&\frac{t}{\log^{d_2-1}(N)}\sum_{k=0}^{d-1}\frac{\log^k(\frac{N^d\log^{d_2 -1}(N)}{t})}{k!}\\
&\geq&\frac{t}{\log^{d_2-1}(N)}\cdot\frac{\log^{d-1}(\frac{N^d\log^{d_2-1}(N)}{t})}{(d-1)!}\\
&\stackrel{t\in(0,1)}{\geq}&\frac{t}{\log^{d_2-1}(N)}\cdot\frac{\log^{d-1}(N^d\log^{d_2-1}(N))}{(d-1)!}\\
&\stackrel{N\geq3}{\geq} &\frac{t}{\log^{d_2-1}(N)}\cdot\frac{d^{d-1}\log^{d-1}(N)}{(d-1)!}\geq t\log^{d_1}(N).
\end{eqnarray*}
It follows that from Step 2 that
\begin{equation}\label{step-3-equ}\lim_{N\to\infty}\int_{\psi_{N,d}(\frac{t}{\log^{d_2-1}(N)})}^{N^d}\mu(w,f)\Big|\log\big(\frac{\psi_{N,d}^{-1}(w)\log^{d_2-1}(N)}{t}\big)\Big|^{d-1}\frac{(\psi_{N,d}^{-1})'(w)}{\psi_{N,d}^{-1}(w)}dw=0.
	\end{equation}
Substituting $w=\psi_{N,d}(v),$ we infer the claim in Step 3.

{\bf Step 4:} The function $\psi_{d_2}$ defined by \eqref{psi_d-function} is sub-multiplicative on $(0,1).$ 
Indeed, for $0<s,t<1,$ put
$$
x=\log\frac1s,\qquad y=\log\frac1t.
$$
Then $x,y>0,$ $s=e^{-x},$ $t=e^{-y},$ and by \eqref{psi_d-function},
$$
\psi_{d_2}(s)
=
e^{-x}\sum_{j=0}^{d_2-1}\frac{x^j}{j!},
\qquad
\psi_{d_2}(t)
=
e^{-y}\sum_{\ell=0}^{d_2-1}\frac{y^\ell}{\ell!}.
$$
Moreover,
$$
\psi_{d_2}(st)
=
e^{-(x+y)}
\sum_{k=0}^{d_2-1}\frac{(x+y)^k}{k!}.
$$
Using the binomial identity
$$
\frac{(x+y)^k}{k!}
=
\sum_{j=0}^k
\frac{x^j}{j!}\frac{y^{k-j}}{(k-j)!},
$$
we obtain
$$
\begin{aligned}
	\psi_{d_2}(st)
	&=
	e^{-(x+y)}
	\sum_{k=0}^{d_2-1}
	\sum_{j=0}^k
	\frac{x^j}{j!}\frac{y^{k-j}}{(k-j)!}\\
	&\leq
	e^{-(x+y)}
	\left(\sum_{j=0}^{d_2-1}\frac{x^j}{j!}\right)
	\left(\sum_{\ell=0}^{d_2-1}\frac{y^\ell}{\ell!}\right)\\
	&=
	\psi_{d_2}(s)\psi_{d_2}(t).
\end{aligned}
$$
Therefore,
$$
	\psi_{d_2}(st)\leq
	\psi_{d_2}(s)\psi_{d_2}(t),
	\quad 0<s,t<1,
$$
and hence, $\psi_{d_2}$ is sub-multiplicative on $(0,1).$ Hence, the function $\psi_{d_2}^{-1}$ is super-multiplicative on $(0,1).$ In other words,
$$N^{d_1}\psi_{N,d_2}^{-1}(\frac{t}{N^{d_1}})=N^d\psi_{d_2}^{-1}(\frac{t}{N^d})\geq \psi_{d_2}^{-1}(t)\cdot N^d\psi_{d_2}^{-1}(N^{-d})\geq \frac{c_d\psi_{d_2}^{-1}(t)}{\log^{d_2-1}(N)}$$
for every $t\in(0,1).$ Here, without loss of generality, $c_d\in(0,1).$
Substituting $c_d\psi_{d_2}^{-1}(t)$ instead of $t$ in Step 3, we obtain
$$\lim_{N\to\infty}\int_{\frac{c_d\psi_{d_2}^{-1}(t)}{\log^{d_2-1}(N)}}^{N^d}\mu(\psi_{N,d}(v),f)\log^{d-1}\big(\frac{v\log^{d_2-1}(N)}{c_d\psi_{d_2}^{-1}(t)}\big)\frac{dv}{v}=0.$$
Combining this with equation \eqref{step-3-equ}, we obtain
$$\lim_{N\to\infty}\int_{N^{d_1}\psi_{N,d_2}^{-1}(\frac{t}{N^{d_1}})}^{N^d}\frac{\mu(\psi_{N,d}(v),f)}{v}\log^{d-1}\big(\frac{v}{N^{d_1}\psi_{N,d_2}^{-1}(\frac{t}{N^{d_1}})}\big)dv=0$$
for every $t\in(0,1).$ 

Since the expression is decreasing in $t,$ the assertion follow for an arbitrary $t>0$ (take $\min\{t,\frac12\}$ instead of $t$ if $t\geq1$).
\end{proof}

We are now fully equipped to proceed with the proof of Theorem \ref{c tensor power thm}.
\begin{proof}[Proof of Theorem \ref{c tensor power thm}] The first assertion in the theorem follows by combining  Lemma \ref{lem:C-rearrangement} and Lemma \ref{first convergence lemma}.

We concentrate on the second assertion. By Lemma \ref{second main equality lemma} (using the notations in Lemma \ref{main mu Cast lemma}), 
$$F_{N,\emptyset}=(C^{\ast})^{\otimes d}f_{N,d}+\sum_{\emptyset\neq S\subsetneq\{1,\dots,d\}}(-1)^{|S|-1}F_{N,S},$$
where $$F_{N,S}:(t_1,\dots,t_d)\to \int_{N^{|S|}\prod_{i\notin S}t_i}^{N^d}\frac{\mu(\psi_{N,d}(v),f)}{v}\frac{\log^{d-1}\big(\frac{v}{N^{|S|}\prod_{i\notin S}t_i}\big)}{(d-1)!}dv.$$
For a given $\epsilon>0,$ we have
$$\mu((1+(2^d-1)\epsilon)t,F_{N,\emptyset})\leq \mu(t,(C^{\ast})^{\otimes d}f_{N,d})+\sum_{\emptyset\neq S\subsetneq\{1,\dots,d\}}\mu(\epsilon t,F_{N,S}).$$
Therefore,
$$\liminf_{N\to\infty}\mu((1+(2^d-1)\epsilon)t,F_{N,\emptyset})\leq$$
$$\leq\liminf_{N\to\infty}\mu(t,(C^{\ast})^{\otimes d}f_{N,d})+\sum_{\emptyset\neq S\subsetneq\{1,\dots,d\}}\limsup_{N\to\infty}\mu(\epsilon t,F_{N,S}).$$
Since $\mu(f)\in {\rm dom}(C^d)\cap{\rm dom}(C^{*d}),$ assumption \eqref{log-tail-assumption} holds automatically. Therefore, using Lemmas \ref{second convergence lemma} and \ref{third convergence lemma}, we conclude
$$((C^{\ast})^d\mu(f))((1+(2^d-1)\epsilon)t)\leq\liminf_{N\to\infty}\mu(t,(C^{\ast})^{\otimes d}f_{N,d}).$$
Passing $\epsilon\to0,$ we complete the proof. 
\end{proof}

\section{Proof of Theorem \ref{main thm}}

This section is devoted to the proof of the main result, Theorem \ref{main thm}. Before doing so, we establish several auxiliary results that will be required in the argument.

The following well-known lemma establishes a fundamental relation between the operators $C,$ $C^*,$ and $S.$ We include its proof for the reader’s convenience.
\begin{lem}\label{kanat lemma} Let $C,$ $C^*,$ and $S$ be the operators defined by \eqref{Ces}, \eqref{Cesaro-adjoint}, and \eqref{S}, respectively. For any $0\leq f\in\Lambda_{\log}(0,\infty)$ we have 
	$$C^{\ast}Cf=CC^{\ast}f=Sf.$$
\end{lem}
\begin{proof} Let $0\leq f\in \Lambda_{\log}(0,\infty).$ By formulas \eqref{Ces} and \eqref{Cesaro-adjoint}, we write 
	$$(C^{\ast}Cf)(t)\stackrel{\eqref{Cesaro-adjoint},\eqref{Ces} }{=}\int_t^{\infty}\Big(\frac1s\int_0^sf(\xi)d\xi\Big)\frac{ds}{s}$$	
	Using the Fubini theorem, 
	$$RHS=\int_0^tf(\xi)\Big(\int_t^{\infty}\frac{ds}{s^2}\Big)d\xi+\int_t^{\infty}f(\xi)\Big(\int_{\xi}^{\infty}\frac{ds}{s^2}\Big)d\xi$$
	$$=\frac{1}{t}\int_{0}^{t}f(\xi)d\xi+\int_{s}^{\infty}\frac{f(\xi)}{\xi}d\xi\stackrel{\eqref{S}}{=}(Sf)(t)$$
	and	
	$$(CC^{\ast}f)(t)\stackrel{\eqref{Ces},\eqref{Cesaro-adjoint} }{=}\frac{1}{t}\int_{0}^{t}\Big(\int_{s}^{\infty}\frac{f(\xi)}{\xi}d\xi\Big)ds.$$
	Again using the Fubini theorem, we have 
	$$RHS=\frac{1}{t}\int_{0}^{t}\frac{f(\xi)}{\xi}\Big(\int_0^{\xi} ds\Big)d\xi+\frac{1}{t}\int_{s}^{\infty}\frac{f(\xi)}{\xi}\Big(\int_0^tds\Big)d\xi$$
	$$=\frac{1}{t}\int_{0}^{t}f(\xi)d\xi+\int_{s}^{\infty}\frac{f(\xi)}{\xi}d\xi\stackrel{\eqref{S}}{=}(Sf)(t).$$
\end{proof}

We next establish a combinatorial identity involving powers of the Ces\`{a}ro operator and its adjoint, which will be useful in iterated operator computations.
\begin{lem} \label{S power lemma} For $k\in \mathbb{N}$ let $C^k, C^{\ast k},$ and $S$ be the operators defined by \eqref{Cesaro d power}, \eqref{Ces adj d power}, and \eqref{S}, respectively.
	For every $f\in{\rm dom}(S^{k+1}),$ we have
	$$S(C^k+C^{\ast k})f=2Sf+\sum_{l=2}^{k+1}(C^l+C^{\ast l})f.$$
\end{lem}
\begin{proof} We prove the assertion by induction on $k.$ Base of induction (i.e., the case $k=1$) follows from Lemma \ref{kanat lemma}. Indeed, if $k=1,$ then by Lemma \ref{kanat lemma}
	$$S(C+C^{\ast})f\stackrel{\eqref{S}}{=}(C+C^{\ast})(C+C^{\ast})f=C^2f+C^{* 2}f+CC^{*}f+C^{*}Cf$$
	$$=C^2f+C^{* 2}f+2CC^{*}f=2Sf+C^2f+C^{* 2}f.$$
	Let us now establish the step of induction. Suppose the assertion holds for $k.$ Let us prove it for $k+1.$ We have
\begin{eqnarray*}S(C^{k+1}+C^{\ast k+1})f&=&(C+C^*)(C^{k+1}+C^{\ast k+1})f\\
&=&(C^{k+2}+CC^{\ast k+1}+C^{\ast}C^{k+1}+C^{\ast k+2})f\\
&=&(C^{k+2}+SC^{\ast k}+SC^k+C^{\ast k+2})f\\
&=&(C^{k+2}+C^{\ast k+2})f+2Sf+\sum_{l=2}^{k+1}(C^l+C^{\ast l})f\\
&=&2Sf+\sum_{l=2}^{k+2}(C^l+C^{\ast l})f.
	\end{eqnarray*}
\end{proof}

The following lemma decomposes the
$d$-th power of $S$ into a sum of Ces\`{a}ro operator powers, providing a key combinatorial tool for subsequent estimates.
\begin{lem}\label{post-kanat equality}  Let $d\in \mathbb{N}$ and let $C^d, C^{\ast d},$ and $S^d$ be the operators defined by \eqref{Cesaro d power}, \eqref{Ces adj d power}, and  \eqref{S d power}, respectively. For every $f\in{\rm dom}(S^d),$ we have
	$$S^df=(\sum_{k=1}^da_{d,k}(C^k+C^{\ast k})f,$$
	where $a_{d,k}$ is a positive integer dependent on $d$ and $k$ only.
\end{lem}
\begin{proof} We prove the assertion by induction on $d.$ Base of induction (i.e., the case $d=1,$ is obvious, $a_{1,1}=1$). Let us now establish the step of induction. Suppose the assertion holds for $d.$ Let us prove it for $d+1.$ We have
	$$S^{d+1}f=S(\sum_{k=1}^da_{d,k}(C^k+C^{\ast k}f).$$
	Now, by Lemma \ref{S power lemma}, we have
	$$S(C^k+C^{\ast k})=2Sf+\sum_{l=2}^{k+1}(C^l+C^{\ast l})f.$$
	Hence,
	$$S^{d+1}f=2(\sum_{k=1}^da_{d,k})Sf+\sum_{k=1}^da_{d,k}\sum_{l=2}^{k+1}(C^l+C^{\ast l})f$$
	$$=2(\sum_{k=1}^da_{d,k})Sf+\sum_{l=2}^{d+1}\big(\sum_{k=l-1}^da_{d,k}\big)(C^l+C^{\ast l})f.$$
	This yields the step of induction and, hence, completes the proof.
\end{proof}

The next lemma shows the main inequality between the operators $C^d, C^{\ast d},$ and $S^d.$
\begin{lem}\label{post-kanat lemma}  Let $d\in \mathbb{N}$ and let $C^d, C^{\ast d},$ and $S^d$ be the operators defined by \eqref{Cesaro d power}, \eqref{Ces adj d power}, and \eqref{S d power}, respectively. For any $f=\mu(f)\in{\rm dom}(S^d),$ we have
	$$S^df\leq c_d\Big(C^d+C^{\ast d}\Big)f,$$
	where $c_d>0$ is a constant depending on $d$ only.
\end{lem}
\begin{proof} By Lemma \ref{post-kanat equality}, we have
	\begin{equation}\label{S^d-equality}S^df=(\sum_{k=1}^da_{d,k}(C^k+C^{\ast k})f.
	\end{equation}
	
	If $f=\mu(f),$ then $1\leq k\leq d$
\begin{eqnarray*}\int_t^{\infty}f(s)\log^{k-1}(\frac{s}{t})\frac{ds}{s}&=& \int_{et}^{\infty}f(s)\log^{k-1}(\frac{s}{t})\frac{ds}{s}+\int_{t}^{et}f(s)\log^{k-1}(\frac{s}{t})\frac{ds}{s}\\
	&\leq& f(t)\cdot\int_{t}^{et}\log^{k-1}(\frac{s}{t})\frac{ds}{s} +\int_{et}^{\infty}f(s)\log^{k-1}(\frac{s}{t})\frac{ds}{s}\\
&=&f(t)\cdot\int_{t}^{et}\log^{k-1}(\frac{s}{t})d\log(\frac{s}{t}) +\int_{et}^{\infty}f(s)\log^{k-1}(\frac{s}{t})\frac{ds}{s}\\
&=&\frac{1}{k}\cdot f(t) +\int_{et}^{\infty}f(s)\log^{k-1}(\frac{s}{t})\frac{ds}{s}\leq f(t) +\int_{et}^{\infty}f(s)\log^{k-1}(\frac{s}{t})\frac{ds}{s}\\
&\leq&f(t) +\int_{et}^{\infty}f(s)\log^{d-1}(\frac{s}{t})\frac{ds}{s}\leq f(t) +\int_{t}^{\infty}f(s)\log^{d-1}(\frac{s}{t})\frac{ds}{s}.
	\end{eqnarray*}
	In other words, we have
	\begin{equation}\label{Ces adj power est}(C^{\ast k}f)(t)\stackrel{\eqref{Ces adj d power}}{=}\frac{1}{\Gamma(k)}\int_t^{\infty}f(s)\log^{k-1}(\frac{s}{t})\frac{ds}{s}\leq \frac{1}{\Gamma(k)}f(t)+\frac{\Gamma(d)}{\Gamma(k)}(C^{\ast d}f)(t), \,\ t>0.
	\end{equation}
	Since $f=\mu(f)$ is decreasing we always have 
	$$Cf(t)\stackrel{\eqref{Ces}}{=}\frac{1}{t}\int_0^t f(s)ds\geq \frac{1}{t}f(t)\cdot\int_0^tds=f(t),\,\ t>0.$$
	It follows that 
	\begin{equation}\label{Cesaro power est}C^kf\leq C^df,\quad 0\leq k\leq d.
	\end{equation} 
	Thus,
	$$S^df\stackrel{\eqref{S^d-equality}}{=}\Big(\sum_{k=1}^da_{d,k}(C^k+C^{\ast k})\Big)f\stackrel{\eqref{Ces adj power est}}{\leq}\sum_{k=1}^da_{d,k}C^kf+\sum_{k=1}^d a_{d,k}\frac{\Gamma(d)}{\Gamma(k)}C^{\ast d}f+\sum_{k=1}^d a_{d,k}\frac{1}{\Gamma(k)}f $$
	$$\stackrel{\eqref{Cesaro power est}}{\leq} \big(\sum_{k=1}^da_{d,k}+\sum_{k=1}^d\frac{a_{d,k}}{\Gamma(k)}\Big)C^df+\sum_{k=1}^d a_{d,k}\frac{\Gamma(d)}{\Gamma(k)}C^{\ast d}f\leq c_d\Big(C^d+C^{\ast d}\Big)f,$$
	where 
	$$c_d=2\Gamma(d)\sum_{k=1}^d\frac{a_{d,k}}{\Gamma(k)}$$
	and $a_{d,k}$ is as in Lemma \ref{post-kanat equality}.
\end{proof}
	Let
$$
\mathcal D_0(\mathbb R)
:=
\left\{
g\in L_\infty(\mathbb R):
\operatorname{supp}(g)\text{ is bounded}
\right\}
$$
and define
$$
\mathcal A_d
	:=
	\left\{
	\sum_{l=1}^n
	f_{l,1}\otimes\cdots\otimes f_{l,d}:
	n\in\mathbb N,\quad
	f_{l,j}\in\mathcal D_0(\mathbb R)
	\right\}.
$$
Here, an elementary tensor
$$
f_1\otimes\cdots\otimes f_d
$$
is identified with the function on  $\mathbb R^d$  defined by
$$
(f_1\otimes\cdots\otimes f_d)(x_1,\dots,x_d)
=
f_1(x_1)\cdots f_d(x_d).
$$
Thus, $f\in\mathcal A_d$ if and only if it admits a representation
$$
f(x_1,\ldots,x_d)
=
\sum_{l=1}^n\prod_{j=1}^d f_{l,j}(x_j),
\quad f_{l,j}\in\mathcal D_0(\mathbb R).
$$
The space $\mathcal A_d$ is dense in $L_p(\mathbb R^d)$ for
every $1\leq p<\infty.$ Indeed, finite linear combinations of
characteristic functions of bounded rectangles are dense in
$L_p(\mathbb R^d)$; see \cite[Proposition~6.7]{Folland}. Moreover,
$$
\chi_{I_1\times\cdots\times I_d}
=
\chi_{I_1}\otimes\cdots\otimes\chi_{I_d}
\in\mathcal A_d,
$$
where $I_j,\, j=1,2,\dots,d,$ are bounded intervals in $\mathbb R.$
For $0\leq k<d,$ define
$$
\mathbf H_k:
=
{\rm id}^{\otimes k}\otimes H\otimes
{\rm id}^{\otimes(d-1-k)}
$$
initially on $\mathcal A_d$ by
\begin{equation}\label{b2}
\mathbf H_k(f_1\otimes\cdots\otimes f_d)
=
f_1\otimes\cdots\otimes f_k\otimes Hf_{k+1}
\otimes f_{k+2}\otimes\cdots\otimes f_d,
\end{equation}
and extend this definition linearly.
Equivalently, if $s=x_{k+1}$,  $x'$ denotes the remaining
coordinates, and
$$
f_{x'}(s):=
f(x_1,\ldots,x_k,s,x_{k+2},\ldots,x_d),
$$
then
$$
(\mathbf H_kf)(x',s)=H(f_{x'})(s)
$$
for almost every $(x',s)\in\mathbb R^{d-1}\times\mathbb R.$
This fibrewise representation shows that $\mathbf H_k$ is independent of
the representation of $f$ as a finite sum of elementary tensors. Indeed, if two finite tensor representations determine the same
function $f$ almost everywhere on $\mathbb R^d,$ then Fubini's
theorem implies that their sections agree almost everywhere on
$\mathbb R$ for almost every $x'.$ Their Hilbert transforms
therefore agree almost everywhere. Hence $\mathbf H_kf$ is
independent of the chosen tensor representation.

The following is the key lemma which shows the bounded extensions of the operator defined by \eqref{b2}.
\begin{lem}\label{STZ-assumption}
	Let $d\geq1$ and $0\leq k<d.$ The operator $\mathbf H_k,$ initially
	defined on $\mathcal A_d$ by \eqref{b2}, admits a unique bounded
	extension
$$
\mathbf H_k:L_1(\mathbb R^d)\rightarrow L_{1,\infty}(\mathbb R^d).
$$
	Moreover, for every $1<p<\infty,$ it admits a unique bounded
	extension
$$
\mathbf H_k:L_p(\mathbb R^d)\rightarrow L_p(\mathbb R^d),
$$
	and
$$
	\|\mathbf H_k\|_{L_p(\mathbb R^d)\to L_p(\mathbb R^d)}
	=
	\|H\|_{L_p(\mathbb R)\to L_p(\mathbb R)}.
$$
\end{lem}
	\begin{proof}
		For $d=1,$ the assertion follows from the classical Riesz \cite[Theorem 4.9 (a), p.139]{BSh} and  Kolmogorov \cite[Theorem 4.9 (b), p.139]{BSh} theorems. Assume that $d\geq2.$ Write
$$
		x=(x',s)\in\mathbb R^{d-1}\times\mathbb R,
		\quad s=x_{k+1},
$$
where $x'$ denotes the remaining $d-1$ variables. For a
		measurable function $f$ on $\mathbb R^d,$ write
$$
		f_{x'}(s):=f(x',s).
$$
		For $f\in\mathcal A_d,$ we have
$$
		(\mathbf H_kf)(x',s)=H(f_{x'})(s)
$$
		for almost every $x'\in\mathbb R^{d-1}$ and almost every
$s\in\mathbb R.$  By Fubini's theorem,
$$
		f_{x'}\in L_1(\mathbb R)
$$
		for almost every \(x'\), and
$$
		\int_{\mathbb R^{d-1}}\|f_{x'}\|_{L_1(\mathbb R)}\,dx'
		=
		\|f\|_{L_1(\mathbb R^d)}.
$$
Applying Kolmogorov's theorem \cite[Theorem 4.9 (b), p.139]{BSh}, for every $\lambda>0,$ we have
$$
	\left|
	\left\{s\in\mathbb R:
	|H(f_{x'})(s)|>\lambda
	\right\}
	\right|
	\leq
	\frac{C_H}{\lambda}
	\|f_{x'}\|_{L_1(\mathbb R)}
$$
	for almost every $x',$ where
$$
	C_H=\|H\|_{L_1(\mathbb R)\to L_{1,\infty}(\mathbb R)}.
$$
	Integrating this estimate with respect to $x',$ we obtain
$$
	\begin{aligned}
		\left|
		\left\{(x',s)\in\mathbb R^d:
		|(\mathbf H_k f)(x',s)|>\lambda
		\right\}
		\right|
		&=
		\int_{\mathbb R^{d-1}}
		\left|
		\left\{s\in\mathbb R:
		|H(f_{x'})(s)|>\lambda
		\right\}
		\right|dx'
		\\
		&\leq
		\frac{C_H}{\lambda}
		\int_{\mathbb R^{d-1}}
		\|f_{x'}\|_{L_1(\mathbb R)}\,dx'
		\\
		&=
		\frac{C_H}{\lambda}
		\|f\|_{L_1(\mathbb R^d)}.
	\end{aligned}
$$
	Consequently,
$$
	\|\mathbf H_k f\|_{L_{1,\infty}(\mathbb R^d)}
	\leq
	C_H\|f\|_{L_1(\mathbb R^d)}.
$$
Let now $f\in L_1(\mathbb R^d).$ Since $\mathcal A_d$ is dense
in $L_1(\mathbb R^d),$ there exists a sequence
$\{f_n\}_{n\geq1}\subset\mathcal A_d$ such that
$$
\|f_n-f\|_{L_1(\mathbb R^d)}\rightarrow0, \,\ n\to \infty.
$$
Applying the preceding estimate to \(f_n-f_m\), we obtain
$$
\|\mathbf H_kf_n-\mathbf H_kf_m\|_{L_{1,\infty}(\mathbb R^d)}
\leq
C_H\|f_n-f_m\|_{L_1(\mathbb R^d)}.
$$
Thus $\{\mathbf H_kf_n\}_{n\geq1}$ is a Cauchy sequence in the
complete quasi-Banach space $L_{1,\infty}(\mathbb R^d).$
Consequently, there exists $g\in L_{1,\infty}(\mathbb R^d)$ such
that
$$
\mathbf H_kf_n\rightarrow g
\quad\text{in }L_{1,\infty}(\mathbb R^d).
$$
We define
$$
\mathbf H_kf:=g.
$$
The same estimate applied to the difference of two approximating
sequences shows that this definition is independent of the chosen
sequence.

Moreover, by passing, if necessary, to an almost everywhere
convergent subsequence and using the Fatou norm property (see Preliminaries on p.4) of
$L_{1,\infty},$ we obtain
$$
\|\mathbf H_kf\|_{L_{1,\infty}(\mathbb R^d)}
\leq
C_H\|f\|_{L_1(\mathbb R^d)}.
$$
Therefore, $\mathbf H_k$ admits a unique bounded linear extension
from $L_1(\mathbb R^d)$ to $L_{1,\infty}(\mathbb R^d).$

	Now, let $1<p<\infty.$ By Fubini's theorem and the classical
Riesz's theorem \cite[Theorem 4.9 (a), p.139]{BSh},
$$
	\begin{aligned}
		\|\mathbf H_k f\|_{L_p(\mathbb R^d)}^p
		&=
		\int_{\mathbb R^{d-1}}
		\|H(f_{x'})\|_{L_p(\mathbb R)}^p\,dx'
		\\
		&\leq
		\|H\|_{L_p(\mathbb R)\to L_p(\mathbb R)}^p
		\int_{\mathbb R^{d-1}}
		\|f_{x'}\|_{L_p(\mathbb R)}^p\,dx'
		\\
		&=
		\|H\|_{L_p(\mathbb R)\to L_p(\mathbb R)}^p
		\|f\|_{L_p(\mathbb R^d)}^p.
	\end{aligned}
$$
Thus, for every $f\in\mathcal A_d,$
$$
\|\mathbf H_kf\|_{L_p(\mathbb R^d)}
\leq
\|H\|_{L_p(\mathbb R)\to L_p(\mathbb R)}
\|f\|_{L_p(\mathbb R^d)}.
$$
Since $\mathcal A_d$ is dense in $L_p(\mathbb R^d),$
$\mathbf H_k$ admits a unique bounded extension to
$L_p(\mathbb R^d),$ and
$$
\|\mathbf H_k\|_{L_p(\mathbb R^d)\to L_p(\mathbb R^d)}
\leq
\|H\|_{L_p(\mathbb R)\to L_p(\mathbb R)}.
$$
	To prove the converse inequality, choose $g_n\in L_p(\mathbb R)$
	such that
$$
	\|g_n\|_{L_p(\mathbb R)}=1
$$
	and
$$
	\|Hg_n\|_{L_p(\mathbb R)}
	\rightarrow
	\|H\|_{L_p(\mathbb R)\to L_p(\mathbb R)}.
$$
	Define
$$
	G_n(x_1,\dots,x_d)
	=
	g_n(x_{k+1})
	\prod_{\substack{1\leq j\leq d\\j\neq k+1}}
	\chi_{[0,1]}(x_j).
$$
	Then
$$
	\|G_n\|_{L_p(\mathbb R^d)}=1
$$
	and
$$
	(\mathbf H_kG_n)(x_1,\dots,x_d)
	=
	(Hg_n)(x_{k+1})
	\prod_{\substack{1\leq j\leq d\\j\neq k+1}}
	\chi_{[0,1]}(x_j).
$$
	Therefore,
$$
	\|\mathbf H_kG_n\|_{L_p(\mathbb R^d)}
	=
	\|Hg_n\|_{L_p(\mathbb R)}.
$$
	Passing to the limit gives
$$
	\|\mathbf H_k\|_{L_p(\mathbb R^d)\to L_p(\mathbb R^d)}
	\geq
	\|H\|_{L_p(\mathbb R)\to L_p(\mathbb R)}.
$$
	Combining the two inequalities proves
$$
	\|\mathbf H_k\|_{L_p(\mathbb R^d)\to L_p(\mathbb R^d)}
	=
	\|H\|_{L_p(\mathbb R)\to L_p(\mathbb R)}.
$$
\end{proof}

The following lemma shows the main inequality between the tensor powers of the operators $C, C^*,$ and $S.$

\begin{lem}\label{lem:Cesaro-tensor-bound}
Let $d\in \mathbb{N}$. Let $C^{\otimes d},$ $C^{*\otimes d},$ and $S^{\otimes d}$ be the operators defined by \eqref{Ces-d-tensor}, \eqref{Ces-adjoint-d-tensor}, and \eqref{S-d-tensor}, respectively. For a nonnegative measurable function $f\in {\rm dom}(S^d),$ we have the pointwise inequalities
$$
C^{\otimes d} f \;\leq\; S^{\otimes d} f\quad\text{and} \quad  C^{*\otimes d} f \;\leq\; S^{\otimes d} f.
$$
\end{lem}
\begin{proof}
Fix $d\in \mathbb{N}$. For $s=(s_1,\dots,s_d),t=(t_1,\dots,t_d)\in(0,\infty)^d$ and for each coordinate $k=1,\dots,d$ set
$$
A_k(t,s):=\frac{1}{t_k}\chi_{\{s_k\leq t_k\}},\qquad
B_k(t,s):=\frac{1}{s_k}\chi_{\{s_k> t_k\}}.
$$
Since $t_k,s_k \in (0,\infty)$ for all $1\leq k\leq d,$ we have $A_k,B_k\geq0$ and
\begin{equation}\label{s kernel}
	\min\{\tfrac{1}{t_k},\tfrac{1}{s_k}\}=A_k(t,s)+B_k(t,s),\, 1\leq k\leq d.
\end{equation}
Indeed, if $s_k\leq t_k,$ then $B_k(t,s)=0,$
consequently,
$$
\min\{\tfrac{1}{t_k},\tfrac{1}{s_k}\}=\frac{1}{t_k}=A_k(t,s)=A_k(t,s)+B_k(t,s).
$$
If $s_k>t_k,$ then $A_k(t,s)=0,$ and
similarly,
$$
\min\{\tfrac{1}{t_k},\tfrac{1}{s_k}\}=\frac{1}{s_k}=B_k(t,s)=A_k(t,s)+B_k(t,s).
$$
By definitions kernels of these operators satisfy
$$
K_{C^{\otimes d}}(s,t)\stackrel{\eqref{Ces-d-tensor-kernel}}{=}\prod_{k=1}^d A_k(t,s), \quad
K_{C^{*\otimes d}}(s,t)\stackrel{\eqref{Ces-adjoint-d-tensor-kernel}}{=}\prod_{k=1}^d B_k(t,s),
$$
and
$$
K_{S^{\otimes d}}(s,t)\stackrel{\eqref{S-d-tensor-kernel},\eqref{s kernel}}{=}\prod_{k=1}^d\big(A_k(t,s)+B_k(t,s)\big),
$$
respectively. Because all factors are nonnegative, for each $s,t\in(0,\infty)^d$ we have the inequalities
$$
\prod_{k=1}^d A_k(t,s)
\leq\prod_{k=1}^d\big(A_k(t,s)+B_k(t,s)\big)
\quad\text{and}\quad
\prod_{k=1}^d B_k(t,s)
\leq\prod_{k=1}^d\big(A_k(t,s)+B_k(t,s)\big).
$$
Integrating these pointwise kernel inequalities against the nonnegative function $f$ yields, for every fixed $t\in (0,\infty)^d$,
$$
(C^{\otimes d} f)(t)
=\int_{0}^{\infty} \prod_{k=1}^d A_k(t,s)\,f(s)\,ds
\leq \int_{0}^{\infty} \prod_{k=1}^d\big(A_k(t,s)+B_k(t,s)\big)\,f(s)\,ds
= (S^{\otimes d} f)(t),
$$
and similarly
$$
(C^{*\otimes d} f)(t)
\leq (S^{\otimes d} f)(t).
$$
This proves the lemma.
\end{proof}
We are now in a position to proceed with the proof of the main result of this paper.
\begin{proof}[Proof of Theorem \ref{main thm}] Let $f\in S(\mathbb{R}^d)$ be such that $\mu(f)\in {\rm dom}(S^d).$ Denote for brevity let 
	$$\mathbf H_0:=H,\ \  d=1;$$
	$$\mathbf H_0:=H\otimes{\rm id}, \,\  \mathbf H_1:={\rm id}\otimes H, \,\ d=2;$$
$$\mathbf H_k:={\rm id}^{\otimes k}\otimes H\otimes {\rm id}^{\otimes (d-1-k)},\quad 0\leq k<d,\,\ d\geq 3;$$
$$f_k=\mathbf  H_kf_{k+1},\quad 0\leq k<d,\quad f_d=f.$$
Then by Lemma \ref{STZ-assumption}, $\mathbf  H_k$ satisfies the conditions in \cite[Theorem 14]{STZ_JFA}. Therefore, we write
$$\mu(f_k)=\mu(\mathbf  H_kf_{k+1})\leq c_{{\rm abs}}S\mu(f_{k+1}),\quad 0\leq k<d.$$
Since $f_0=H^{\otimes d}f,$ it follows that
$$\mu(H^{\otimes d}f)=\mu(f_0)\leq c_{{\rm abs}}S\mu(f_{1})\leq (c_{{\rm abs}})^2S^2\mu(f_{2})$$
$$\leq \cdots \leq  (c_{{\rm abs}})^dS^d\mu(f_{d})=(c_{{\rm abs}})^dS^d\mu(f).$$
This proves the first assertion (i).	

To prove the second assertion (ii), let $f_{N,d}$ be as in Theorem \ref{c tensor power thm}. Extend $f_{N,d}$ to $\mathbb{R}^d$ by setting $f_{N,d}=0$ on $\mathbb{R}^d\backslash\mathbb{R}^d_+.$ For $t=(t_1,t_2,\dots,t_d)\in\mathbb{R}^d_+,$ we have
$$(H^{\otimes d}f_{N,d})(-t)\stackrel{\eqref{d-tensor-hilbert tr}}{=}\frac{(-1)^d}{\pi^d}\int_{\mathbb{R}_+^d}\frac{f_{N,d}(s)ds}{\prod_{k=1}^d(t_k+s_k)}.$$
Since $f_{N,d}\geq0,$ it follows that
$$|(H^{\otimes d}f_{N,d})(-t)|\geq (2\pi)^{-d}\int_{\mathbb{R}_+^d}\prod_{k=1}^d\min\{t_k^{-1},s_k^{-1}\}f_{N,d}(s)ds\stackrel{\eqref{S-d-tensor}}{=}(2\pi)^{-d}(S^{\otimes d}f_{N,d})(t),\quad t\in\mathbb{R}^d_+.$$
In particular,
$$\mu(H^{\otimes d}f_{N,d})\geq 2^{-d}\mu(S^{\otimes d}f_{N,d}).$$
By Lemma \ref{lem:Cesaro-tensor-bound} we have
$$
S^{\otimes d} f_{N,d}\geq C^{\otimes d} f_{N,d}\quad\text{and} \quad S^{\otimes d} f_{N,d}\geq C^{*\otimes d} f_{N,d},
$$
and consequently,
$$\mu(S^{\otimes d} f_{N,d})\geq \mu(C^{\otimes d} f_{N,d})\quad\text{and} \quad \mu(S^{\otimes d} f_{N,d})\geq \mu(C^{*\otimes d} f_{N,d}),
$$
Thus,
$$\mu(H^{\otimes d}f_{N,d})\geq 2^{-d}\mu(C^{\otimes d}f_{N,d}),\quad \mu(H^{\otimes d}f_{N,d})\geq 2^{-d}\mu(C^{\ast \otimes d}f_{N,d}).$$
By Theorem \ref{c tensor power thm},
$$\liminf_{N\to\infty}\mu(H^{\otimes d}f_{N,d})\geq 2^{-d}C^d\mu(f),\quad \liminf_{N\to\infty}\mu(H^{\otimes d}f_{N,d})\geq 2^{-d}C^{\ast d}\mu(f).$$
Then combining these we obtain
$$\liminf_{N\to\infty}\mu(H^{\otimes d}f_{N,d})\geq 2^{-d-1}(C^d\mu(f)+C^{\ast d}\mu(f)).$$
Hence, by Lemma \ref{post-kanat lemma} we obtain
$$\liminf_{N\to\infty}\mu(H^{\otimes d}f_{N,d})\geq 2^{-d-1}(C^d\mu(f)+C^{\ast d}\mu(f))\geq \frac{1}{c_d\cdot2^{d+1}}S^d\mu(f),$$
which proves the second assertion.
\end{proof}

\subsection{Conflict of Interest Statement:} The author declares that there is no conflict of interest.

\subsection{ Data Availability Statement:} No datasets were generated or analyzed during the current study.

\section{Acknowledgements}
Authors would like to thank Professor Alexei Ber for his detailed discussion and his comments on this work.
The work was partially supported by the Australian Research Council. F.S and K.T. were partially supported by ARC grant FL17010005 and ARC grant DP230100434.  
The authors thank the anonymous referee for the careful reading of the manuscript and for bringing to our attention an error in Lemma 12 of the previous version.


\end{document}